\documentclass[12pt]{amsart}
\usepackage{amsmath, amsthm, amssymb}
\usepackage{geometry}
\usepackage{hyperref}
\usepackage{enumerate}
\usepackage{cite}
\usepackage{ytableau}
\usepackage{adjustbox}
\usepackage{tikz}
\usepackage{float}
\usepackage{bm}
\usepackage{longtable}
\usepackage{enumitem}
\usepackage{mathtools}
\usetikzlibrary{decorations.pathreplacing}
\numberwithin{equation}{section} 
\newtheorem{theorem}{Theorem}[section]
\newtheorem*{maintheorem}{Main Theorem}
\newtheorem{lemma}[theorem]{Lemma}
\newtheorem{proposition}[theorem]{Proposition}

\theoremstyle{definition}
\newtheorem{definition}[theorem]{Definition}
\theoremstyle{remark}
\newtheorem{remark}[theorem]{Remark}
\newtheorem{example}[theorem]{Example}

\newcommand{\shiftnoarrow}[2]{\ensuremath \raisebox{#1cm}{${#2}$}}
\newcommand{\YoungWithData}[2]{%
	\begin{tabular}{@{}c@{}}
		#1 \\[-0.6em]
		{\small $#2$}
	\end{tabular}
}

\title{Verma Bases for finite dimensional Representations of the orthosymplectic Lie superalgebra $\mathfrak{spo}(4|1)$}

\author{Bintao Cao$^1$}
\address{1. School of Mathematics and Statistics, Yunnan University, Kunming 650500, China}
\email{btcao@ynu.edu.cn}

\author{Ye Huang$^2$}
\address{2. School of Mathematics and Statistics, Yunnan University, Kunming 650500, China}
\email{huangye2461@163.com}

\begin{document}
	
	\begin{abstract}
			We define the Verma vector system for each finite dimensional irreducible representation of the orthosymplectic Lie superalgebra $\mathfrak{spo}(4|1)$ with the highest weight $\lambda,$ via the conditions that making a tableau with shape $\lambda$ to be a Kashiwara-Nakashima tableau.
		We then show the linearly 
		independence of this vector system.
		It turns out to be a basis of the finite dimensional irreducible representation $L(\lambda)$ of the orthosymplectic Lie superalgebra $\mathfrak{spo}(4|1)$ with the highest weight $\lambda,$
		which analogs to the Verma basis of representations of $\mathfrak{sp}_4,$ called the Verma basis of the finite dimensional irreducible representation of $\mathfrak{spo}(4|1)$.
	\end{abstract}
	
	\maketitle
	
	\section{Introduction}
	
This manuscript	is inspired by our previous work \cite{CaoHuang},
in which 
we studied Verma bases of finite dimensional irreducible modules of the finite dimensional simple Lie algebra $\mathfrak{sp}_4.$ 
We gave a natural one-to-one correspondence between the set of Verma vectors and the set of Kashiwara-Nakashima tableaux
of $\mathfrak{sp}_4.$ 
We also gave a direct proof of the linear 
independence of the system of Verma vectors.

As we have pointed out in \cite{CaoHuang}, 
finding some ``good'' bases is 
an important problem in representation theory. 
Various good bases have already been found,
and have been used to deal with all kinds of problems in the representation theory of Lie groups, Lie algebras, quantum groups and Hecke algebras.
For example, canonical and dual canonical bases,
introduced by Kazhdan-Lusztig\cite{KazhdanLusztig1979} can give the characteristic formulas of some important modules in
BGG category $\mathcal O.$
Crystal bases, introduced by Lusztig~\cite{lusztig1990canonical} and Kashiwara~\cite{kashiwara1990crystalizing,kashiwara1991}, contain lots of information of the integrable
modules at $q=0$ for the quantized enveloping algebras of some symmetrizable Kac-Moody Lie algebras.
Gelfand-Tsetlin bases, given by Gelfand and Tsetlin~\cite{gelfand1, gelfand2},
are used to study the branching theory.

Verma bases are monomial bases of certain highest weight modules
of some semisimple Lie algebras, expressed in terms of negative simple root vectors. 
The Verma basis was first introduced by
Li, Moody, Nicolescu, and Patera~\cite{li1986verma} in 1986,
based on Verma's method.
They worked out the Verma bases for finite dimensional irreducible representations of the Lie algebras \( A_n (n \geq 1), B_n (2 \leq n \leq 6), C_n (2 \leq n \leq 6), D_n (4 \leq n \leq 6), \) and \( G_2 \). But the proof of the linear independence of these basis vectors is not so clear.
Raghavan and Sankaran~\cite{Raghavan1999} gave a proof for  \( A_n \). 
Po{\v{s}}ta and Havl{\'\i}{\v{c}}ek~\cite{posta2013note} discussed the construction of the Verma bases of the Verma modules of  \( A_n \). 
Hall~\cite{hall1987verma} presented the Verma bases for Verma modules of $A_n$ and some other types of Lie algebras in his Thesis. 
For the application, the  Verma basis has been used in Quantum Chemistry
by Paldus and Planelles~\cite{Paldus2018}.

In this paper, 
we define Verma bases of finite dimensional irreducible modules of the  orthosymplectic Lie superalgebra $\mathfrak{spo}(4|1).$
Analogous to the Lie algebra case, 
they are also monomial bases of certain highest weight modules, expressed in terms of negative simple root vectors.  
The key is that the definition of Verma bases in the Lie algebra $\mathfrak{sp}_4$ case,  depends on certain inequalities.
These inequalities can be explained as just the conditions for tableaux to be Kashiwara-Nakashima tableaux.
In $\mathfrak{spo}(4|1)$ case,
we also have a system of inequalities
that makes a tableau to be a Kashiwara-Nakashima tableau
of $\mathfrak{spo}(4|1).$
Then we can define the Verma bases of the irreducible modules
of $\mathfrak{spo}(4|1)$ by these inequalities naturally. 
After that,
we give a one-to-one correspondence between the set of Verma vectors and the set of Kashiwara-Nakashima tableaux.
At last, we give a proof of the linear 
independence of the system of Verma vectors.
Exactly, we have the following result.

Let $L(\lambda)$ be the finite dimensional irreducible representation of $\mathfrak{spo}(4|1)$ with highest weight $\lambda$. We write $\lambda = m_1\omega_1 + 2m_2\omega_2 = (m_1+m_2)\epsilon_1 + m_2\epsilon_2$ for $m_1, m_2 \in \mathbb{Z}_{\geq 0}$, where $\omega_1, \omega_2$ are the fundamental weights. We also write 
$\lambda=(\lambda_1,\lambda_2)=(m_1+m_2,m_2)$
be a partition for convenience.
Let the nonzero vector $v_{\lambda}$ be the highest weight vector of $L(\lambda)$. 
Then the system of Verma vectors of $L(\lambda)$ 
is defined by
\begin{equation} \label{def:H}
	H:= \left\{ f_1^{b_4} f_2^{b_3} f_1^{b_2} f_2^{b_1} v_\lambda \;\middle|\;
	\begin{aligned}
		b_i &\in \mathbb Z_{\ge0}, \; i=1,2,3,4,\\
		0 &\leq b_1 \leq 2m_2, \\
		0 &\leq b_2 \leq m_1 + b_1, \\
		0 &\leq b_3 \leq \min\{b_2 + m_1,\, 2b_2\}, \\
		0 &\leq b_4 \leq \min\{m_1,\, \tfrac{1}{2}b_3\}
	\end{aligned}
	\right\}, 
\end{equation}
where $f_{1}=E_{21}-E_{34}$ is the root vector corresponding to the negative simple root $\epsilon_2 - \epsilon_1$ of $\mathfrak{spo}(4|1)$, and $f_{2}=E_{52}-E_{45}$ is the root vector corresponding to the negative simple root $-\epsilon_2$.

\begin{maintheorem}
	Let $L(\lambda)$ be the finite dimensional irreducible representation of $\mathfrak{spo}(4|1)$ with highest weight $\lambda$. Let $H$ be the system of Verma vectors of  $L(\lambda)$ (see \eqref{def:H}).
	Then 
	\begin{itemize}
		\item[i).] There exists a weight-preserving one-to-one correspondence between $H$ and the set of Kashiwara-Nakashima tableaux of $\mathfrak{spo}(4|1)$
		with the shape $\lambda$.
		\item[ii).] $H$ is a basis of $L(\lambda).$ 
		
	\end{itemize}
\end{maintheorem}

In Section 2, we introduce the structure of the orthosymplectic Lie superalgebra $\mathfrak{spo}(4|1),$ including the matrix realization and the root system. 
In Section 3,
we introduce the Kashiwara-Nakashima tableaux of $\mathfrak{spo}(4|1).$
Especially, we get a system of inequalities that makes a tableau to be a Kashiwara-Nakashima tableau of $\mathfrak{spo}(4|1)$.
In Section 4, we define the Verma basis of the finite dimensional irreducible module of $\mathfrak{spo}(4|1)$
via the system of inequalities in Section 3.
Then we connect it  to 
the
Kashiwara-Nakashima tableaux of $\mathfrak{spo}(4|1).$
In Section 5, we show the linear independence of the system of Verma vectors.

	\section{Roots and Fundamental Weights of $\mathfrak{spo}(4|1)$}
	In this section, we briefly introduce the matrix realization, the root system, and the fundamental weights of the orthosymplectic Lie superalgebra $\mathfrak{spo}(4|1)$.
	
	In its matrix realization, the orthosymplectic Lie superalgebra $\mathfrak{spo}(4|1)$ \cite{Kac1977} can be defined as the set of all $5 \times 5$ matrices of the following form:
	\begin{equation}\label{matrix}
		\left(
		\begin{array}{cc|cc|c}
			a_{11} & a_{12} & b_{11} & b_{12} & p_1 \\
			a_{21} & a_{22} & b_{12} & b_{22} & p_2 \\
			\hline
			c_{11} & c_{12} & -a_{11} & -a_{21} & q_1 \\
			c_{12} & c_{22} & -a_{12} & -a_{22} & q_2 \\
			\hline
			-q_1 & -q_2 & p_1 & p_2 & 0
		\end{array}
		\right),
	\end{equation}
	where all variables are arbitrary complex numbers. Its even subalgebra $\mathfrak{spo}(4|1)_{\overline{0}} \cong \mathfrak{sp}_4$ consists of the matrices in \eqref{matrix} satisfying $p_1 = p_2 = q_1 = q_2 = 0$, and its odd subspace $\mathfrak{spo}(4|1)_{\overline{1}}$ consists of the matrices in \eqref{matrix} satisfying $a_{11} = a_{12} = a_{21} = a_{22} = b_{11} = b_{12} = b_{22} = c_{11} = c_{12} = c_{22} = 0$.
	
	Let $E_{ij}$ denote the $5 \times 5$ matrix unit with $1$ in the $(i, j)$-entry and $0$ elsewhere. Then the Cartan subalgebra $\mathfrak{h}$ of $\mathfrak{spo}(4|1)$ is spanned by the following diagonal matrices:
	\[ 
	H_1 = E_{11} - E_{33}, \quad H_2 = E_{22} - E_{44}. 
	\]
	Let $\{\epsilon_1, \epsilon_2\}$ be a basis for the dual space $\mathfrak{h}^*$ of $\mathfrak{h}$, satisfying
	\[ 
	\epsilon_i(H_j) = \delta_{ij}, \quad i, j = 1, 2. 
	\]
	Then the root system $\Phi = \Phi_{\bar{0}} \cup \Phi_{\bar{1}}$ of $\mathfrak{spo}(4|1)$ is given by
	\begin{align*}
		\Phi_{\bar{0}} &= \{\pm \epsilon_i\pm\epsilon_j \mid 1 \leq i < j \leq 2\} \cup \{\pm2\epsilon_i \mid 1\leq i\leq 2\},\\
		\Phi_{\bar{1}} &= \{\pm\epsilon_i \mid 1\leq i\leq 2\},
	\end{align*}
	where $\Phi_{\bar{0}}$ and $\Phi_{\bar{1}}$ denote the sets of even and odd roots, respectively. We choose the simple root system to be
	\[ 
	\Pi = \{ \alpha_1 = \epsilon_1 - \epsilon_2, \; \alpha_2 = \epsilon_2 \},
	\]
	where $\alpha_1$ is an even simple root and $\alpha_2$ is an odd simple root. The sets of positive even and positive odd roots are respectively given by:
	\begin{align*}
		\Phi_{\bar{0}}^+ &= \{\epsilon_1 \pm \epsilon_2, 2\epsilon_1, 2\epsilon_2\},\\
		\Phi_{\bar{1}}^+ &= \{\epsilon_1, \epsilon_2\}.
	\end{align*}
	
	The standard Dynkin diagram of $\mathfrak{spo}(4|1)$ is shown in the figure below.
	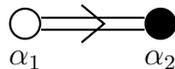
\begin{figure}[htbp]
		\centering
		\begin{tikzpicture}
			\tikzstyle{white node}=[circle, draw=black, thick, minimum size=0.4cm, inner sep=0pt, fill=white]
			\tikzstyle{black node}=[circle, draw=black, thick, minimum size=0.4cm, inner sep=0pt, fill=black]
			
			\node[white node, label=below:$\alpha_1$] (n1) at (0, 0) {};
			\node[black node, label=below:$\alpha_2$] (n2) at (1.8, 0) {};
			
			\draw[thick] ([yshift=0.08cm]n1.east) -- ([yshift=0.08cm]n2.west);
			\draw[thick] ([yshift=-0.08cm]n1.east) -- ([yshift=-0.08cm]n2.west);
			
			\draw[thick] (0.75, 0.25) -- (1.05, 0) -- (0.75, -0.25);
			
		\end{tikzpicture}
		\caption{The standard Dynkin diagram of the Lie superalgebra $\mathfrak{spo}(4|1)$}
	\end{figure}
	
	Let $\omega_1 = \epsilon_1$ and $\omega_2 = \frac{1}{2}(\epsilon_1 + \epsilon_2)$ denote the fundamental weights. Then we have the following theorem.
	
	\begin{theorem}\textup{\cite{Kac1978, Shader1999}}
		The weight $\lambda = a_1\omega_1 + a_2\omega_2$ is a highest weight for a finite dimensional irreducible $\mathfrak{spo}(4|1)$-module if and only if $a_1$ and $\frac{1}{2}a_2$ are non-negative integers.
	\end{theorem}

	\section{Kashiwara-Nakashima Tableaux of $\mathfrak{spo}(4|1)$}
	A partition $\lambda$ of a positive integer $n$ is a $k$-tuple $\lambda=(\lambda_1,\ldots,\lambda_k)$, where $\lambda_1\geq\lambda_2\geq\cdots\geq\lambda_k > 0$ and $|\lambda|=\sum_{i = 1}^k\lambda_i=n$. The length of $\lambda$ is $\ell(\lambda)=k$, and the Young diagram of shape $\lambda$ is given by arranging $n$ boxes in $k$ left-justified rows with $\lambda_i$ boxes in the $i$th row. Figure~\ref{Young diagram} shows a Young diagram of shape $(4,3,1)$.
	
	A semistandard Young tableau of shape $\lambda$ is a filling of the Young diagram of shape $\lambda$ with positive integers such that the entries are weakly increasing across each row and strictly increasing down each column. Figure~\ref{Young tableau} shows a semistandard Young tableau.
	
	\begin{figure}[h]
		\centering
		\begin{minipage}{0.45\textwidth}
			\centering
			\begin{ytableau}
				~ & ~ & ~ & ~ \\
				~ & ~ & ~ \\
				~
			\end{ytableau}
			\caption{Young diagram}
			\label{Young diagram}
		\end{minipage}
		\begin{minipage}{0.45\textwidth}
			\centering
			\begin{ytableau}
				1 & 2 & 3 & 3 \\
				3 & 3 & 4 \\
				4
			\end{ytableau}
			\caption{Young tableau}
			\label{Young tableau}
		\end{minipage}
	\end{figure}
	
	The finite dimensional irreducible representations of $\mathfrak{spo}(4|1)$ are parametrized by their highest weights $\lambda$, which have the form
	\[
	\lambda = m_1\omega_1 + 2m_2\omega_2 = \lambda_1\epsilon_1 + \lambda_2\epsilon_2, \quad \text{with } m_1, m_2 \in \mathbb{Z}_{\geq 0},
	\]
	where $\lambda_1 = m_1 + m_2$ and $\lambda_2 = m_2$. Hence, $\lambda$ corresponds to a partition $(\lambda_1 , \lambda_2)$ with at most two parts.
	
	Kashiwara-Nakashima tableaux (KN tableaux for short) of $\mathfrak{spo}(4|1)$ have entries from the set $\mathcal{N}=\{1, 2, 0, \overline{2}, \overline{1}\}$ with ordering
	\[1 < 2 < 0< \overline{2} < \overline{1}.\]
	
	\begin{definition}\label{def:ospKN}\cite{LiuYang2010}
		Let $\lambda$ be a partition with $\ell(\lambda) \leq 2$. A KN tableau $T$ of $\mathfrak{spo}(4|1)$ of shape $\lambda$ is a filling of the Young diagram of shape $\lambda$ with entries from $\mathcal{N}$, satisfying the following conditions:
		\begin{enumerate}[label=(\roman*), itemsep=0pt]
			\item The entries in $T$ are weakly increasing across each row, with $0$ not allowed to repeat; the entries are strictly increasing down each column, but $0$ may repeat.
			\item The entries $1$ and $\overline{1}$ do not appear in the same column.
			\item If \(T\) has two adjacent columns, then the following two configurations are forbidden:
			\vspace{-1.5em}
			$$\left.\begin{array}{lr}
				&2 \\
				\\
			\end{array}
			\right|  \begin{array}{r}  \cr  \cr \overline{2} \cr  \cr  \end{array}\shiftnoarrow{-0.2}{,} \quad \quad
			\left.\begin{array}{r}
				0  \\
				\\
			\end{array}
			\right|  \begin{array}{r}  \cr
				\overline{2} \cr
			\end{array}\shiftnoarrow{-0.2}{.}$$
		\end{enumerate}
		We denote by $\mathrm{KN}_{\lambda}(4|1)$ the set of KN tableaux of $\mathfrak{spo}(4|1)$ of shape $\lambda$.
	\end{definition}
	
	For a KN tableau $T$ of $\mathfrak{spo}(4|1)$, we define its weight to be
	\begin{align}
		\mathrm{wt}(T) = \sum_{i = 1}^{2} (k_i - k_{\overline{i}}) \epsilon_i,
	\end{align}
	where $k_i$ (respectively, $k_{\overline{i}}$) is the number of $i$'s (respectively, $\overline{i}$'s) appearing in $T$.
	
	Let $L(\lambda)$ be the finite dimensional irreducible representation of $\mathfrak{spo}(4|1)$ with highest weight $\lambda$, where $\lambda = (m_1 + m_2)\epsilon_1 + m_2\epsilon_2$ and $m_1, m_2 \in \mathbb{Z}_{\geq 0}$.
	
	\begin{proposition}\label{dimension}
		$\dim L(\lambda) = |\mathrm{KN}_{\lambda}(4|1)|.$
	\end{proposition}
	\begin{proof}
		In \cite{LiuYang2010}, the KN tableaux of $\mathfrak{spo}(4|1)$ of shape $\lambda$ are used to label the elements of the crystal basis of the $U_q(\mathfrak{spo}(4|1))$-module $L_q(\lambda)$, which is a quantum deformation of $L(\lambda)$. Since there is a one-to-one correspondence between the elements of the crystal basis and the KN tableaux in $\mathrm{KN}_{\lambda}(4|1)$, the number of such tableaux equals the dimension of the representation $L_q(\lambda)$.
		
		Furthermore, because $L_q(\lambda)$ is a quantum deformation of $L(\lambda)$, their dimensions are equal, i.e., $\dim L(\lambda) = \dim L_q(\lambda)$. Consequently, the number of KN tableaux in $\mathrm{KN}_{\lambda}(4|1)$ is exactly the dimension of $L(\lambda)$.
	\end{proof}
	
	Let $\lambda=(m_{1}+m_{2}, m_{2})$ and $T\in \mathrm{KN}_{\lambda}(4|1)$, where $m_{1}, m_{2}\in \mathbb{Z}_{\geq0}$. We define:
	\begin{enumerate}[label=(\roman*), itemsep=0pt]
		\item $b_1$ as the sum of twice the number of entries in the second row of $T$ that are greater than or equal to $\overline{2}$ and the number of entries in the second row equal to $0$;
		\item $b_2$ as the sum of the number of entries in the first row strictly greater than $1$ and the number of entries in the second row strictly greater than $\overline{2}$;
		\item $b_3$ as the sum of twice the number of entries in the first row that are greater than or equal to $\overline{2}$ and the number of entries in the first row equal to $0$;
		\item $b_4$ as the number of entries in the first row strictly greater than $\overline{2}$.
	\end{enumerate}
	From the definitions of $b_i$ ($i=1,2,3,4$), we first obtain the following inequalities:
	\[ 
	0\leq b_{1}\leq 2m_{2}, \quad 0\leq b_{2}\leq m_{1}+b_{1}, \quad 0\leq b_{3}\leq 2b_{2}, \quad 0\leq b_{4}\leq m_{1}. 
	\]
	
	By the definition of $\mathrm{KN}_{\lambda}(4|1)$, the tableau $T$ falls into the following cases.
	
	\begin{figure}[H]
		\begin{center}
			\begin{tikzpicture}[scale=0.7, every node/.style={font=\large}]
				
				\draw (0,0) rectangle ++(4,1);
				\node at (0.5,0.5) {1};
				\node at (1.5,0.5) {$\cdots$};
				\node at (2.5,0.5) {1};
				
				\draw (3,0) rectangle ++(5,1);
				\node at (3.5,0.5) {1};
				\node at (4.5,0.5) {1};
				\node at (5.5,0.5) {$\cdots$};
				\node at (6.5,0.5) {1};
				\node at (7.5,0.5) {1};
				
				\draw (8,0) rectangle ++(10,1);
				\node at (8.5,0.5) {2};
				\node at (9.5,0.5) {$\cdots$};
				\node at (10.5,0.5) {2};
				\node at (11.5,0.5) {0};
				\node at (12.5,0.5) {$\overline{2}$};
				\node at (13.5,0.5) {$\cdots$};
				\node at (14.5,0.5) {$\overline{2}$};
				\node at (15.5,0.5) {$\overline{1}$};
				\node at (16.5,0.5) {$\cdots$};
				\node at (17.5,0.5) {$\overline{1}$};
				
				\draw (0,-1) rectangle ++(4,1);
				\node at (0.5,-0.5) {2};
				\node at (1.5,-0.5) {$\cdots$};
				\node at (2.5,-0.5) {2};
				
				\draw (3,-1) rectangle ++(5,1);
				\node at (3.5,-0.5) {0};
				\node at (4.5,-0.5) {$\overline{2}$};
				\node at (5.5,-0.5) {$\cdots$};
				\node at (6.5,-0.5) {$\overline{2}$};
				\node at (7.5,-0.5) {$\overline{2}$};
				
				\draw (8,-1) rectangle ++(6,1);
				\node at (9,-0.5) {$\overline{1}$};
				\node at (11,-0.5) {$\cdots$};
				\node at (13,-0.5) {$\overline{1}$};

			\end{tikzpicture}
			\caption{} 
			\label{OSP-KN1}
		\end{center}
	\end{figure}
	\begin{figure}[H]
		\begin{center}
			\begin{tikzpicture}[scale=0.7, every node/.style={font=\large}]
				
				\draw (0,0) rectangle ++(4,1);
				\node at (0.5,0.5) {1};
				\node at (1.5,0.5) {$\cdots$};
				\node at (2.5,0.5) {1};
				
				\draw (3,0) rectangle ++(5,1);
				\node at (3.5,0.5) {1};
				\node at (4.5,0.5) {1};
				\node at (5.5,0.5) {$\cdots$};
				\node at (6.5,0.5) {1};
				\node at (7.5,0.5) {2};
				
				\draw (8,0) rectangle ++(10,1);
				\node at (8.5,0.5) {2};
				\node at (9.5,0.5) {$\cdots$};
				\node at (10.5,0.5) {2};
				\node at (11.5,0.5) {0};
				\node at (12.5,0.5) {$\overline{2}$};
				\node at (13.5,0.5) {$\cdots$};
				\node at (14.5,0.5) {$\overline{2}$};
				\node at (15.5,0.5) {$\overline{1}$};
				\node at (16.5,0.5) {$\cdots$};
				\node at (17.5,0.5) {$\overline{1}$};
				
				\draw (0,-1) rectangle ++(4,1);
				\node at (0.5,-0.5) {2};
				\node at (1.5,-0.5) {$\cdots$};
				\node at (2.5,-0.5) {2};
				
				\draw (3,-1) rectangle ++(5,1);
				\node at (3.5,-0.5) {0};
				\node at (4.5,-0.5) {$\overline{2}$};
				\node at (5.5,-0.5) {$\cdots$};
				\node at (6.5,-0.5) {$\overline{2}$};
				\node at (7.5,-0.5) {$\overline{2}$};
				
				\draw (8,-1) rectangle ++(6,1);
				\node at (9,-0.5) {$\overline{1}$};
				\node at (11,-0.5) {$\cdots$};
				\node at (13,-0.5) {$\overline{1}$};

			\end{tikzpicture}
			\caption{} 
			\label{OSP-KN2}
		\end{center}
	\end{figure}
	\begin{figure}[H]
		\begin{center}
			\begin{tikzpicture}[scale=0.7, every node/.style={font=\large}]
				
				\draw (0,0) rectangle ++(4,1);
				\node at (0.5,0.5) {1};
				\node at (1.5,0.5) {$\cdots$};
				\node at (2.5,0.5) {1};
				
				\draw (3,0) rectangle ++(5,1);
				\node at (3.5,0.5) {1};
				\node at (4.5,0.5) {1};
				\node at (5.5,0.5) {$\cdots$};
				\node at (6.5,0.5) {1};
				\node at (7.5,0.5) {0};
				
				\draw (8,0) rectangle ++(10,1);
				\node at (9,0.5) {$\overline{2}$};
				\node at (10,0.5) {$\cdots$};
				\node at (11,0.5) {$\overline{2}$};
				
				\node at (12.5,0.5) {$\overline{2}$};
				\node at (13.5,0.5) {$\cdots$};
				\node at (14.5,0.5) {$\overline{2}$};
				\node at (15.5,0.5) {$\overline{1}$};
				\node at (16.5,0.5) {$\cdots$};
				\node at (17.5,0.5) {$\overline{1}$};
				
				\draw (0,-1) rectangle ++(4,1);
				\node at (0.5,-0.5) {2};
				\node at (1.5,-0.5) {$\cdots$};
				\node at (2.5,-0.5) {2};
				
				\draw (3,-1) rectangle ++(5,1);
				\node at (3.5,-0.5) {0};
				\node at (4.5,-0.5) {$\overline{2}$};
				\node at (5.5,-0.5) {$\cdots$};
				\node at (6.5,-0.5) {$\overline{2}$};
				\node at (7.5,-0.5) {$\overline{2}$};
				
				\draw (8,-1) rectangle ++(6,1);
				\node at (9,-0.5) {$\overline{1}$};
				\node at (11,-0.5) {$\cdots$};
				\node at (13,-0.5) {$\overline{1}$};

			\end{tikzpicture}
			\caption{} 
			\label{OSP-KN3}
		\end{center}
	\end{figure}
	\begin{figure}[H]
		\begin{center}
			\begin{tikzpicture}[scale=0.7, every node/.style={font=\large}]
				
				\draw (0,0) rectangle ++(4,1);
				\node at (0.5,0.5) {1};
				\node at (1.5,0.5) {$\cdots$};
				\node at (2.5,0.5) {1};
				
				\draw (3,0) rectangle ++(4,1);
				\node at (3.5,0.5) {1};
				\node at (4.5,0.5) {1};
				\node at (5.5,0.5) {$\cdots$};
				\node at (6.5,0.5) {1};
				
				\draw (7,0) rectangle ++(11.6,1);
				\node at (7.4,0.5) {1};
				\node at (8.3,0.5) {$\cdots$};
				\node at (9.2,0.5) {1};
				\node at (10.1,0.5) {2};
				\node at (11,0.5) {$\cdots$};
				\node at (11.9,0.5) {2};
				\node at (12.8,0.5) {0};
				\node at (13.7,0.5) {$\overline{2}$};
				\node at (14.6,0.5) {$\cdots$};
				\node at (15.5,0.5) {$\overline{2}$};
				\node at (16.4,0.5) {$\overline{1}$};
				\node at (17.3,0.5) {$\cdots$};
				\node at (18.2,0.5) {$\overline{1}$};
				
				\draw (0,-1) rectangle ++(4,1);
				\node at (0.5,-0.5) {2};
				\node at (1.5,-0.5) {$\cdots$};
				\node at (2.5,-0.5) {2};
				
				\draw (3,-1) rectangle ++(4,1);
				\node at (3.5,-0.5) {0};
				\node at (4.5,-0.5) {$\overline{2}$};
				\node at (5.5,-0.5) {$\cdots$};
				\node at (6.5,-0.5) {$\overline{2}$};
			\end{tikzpicture}
			\caption{} 
			\label{OSP-KN4}
		\end{center}
	\end{figure}
	\begin{figure}[H]
		\begin{center}
			\begin{tikzpicture}[scale=0.7, every node/.style={font=\large}]
				
				\draw (0,0) rectangle ++(4,1);
				\node at (0.5,0.5) {1};
				\node at (1.5,0.5) {$\cdots$};
				\node at (2.5,0.5) {1};
				
				\draw (3,0) rectangle ++(1,1);
				\node at (3.5,0.5) {2};
				
				\draw (4,0) rectangle ++(11,1);
				\node at (5,0.5) {2};
				\node at (6,0.5) {$\cdots$};
				\node at (7,0.5) {2};
				\node at (8,0.5) {0};
				\node at (9,0.5) {$\overline{2}$};
				\node at (10,0.5) {$\cdots$};
				\node at (11,0.5) {$\overline{2}$};
				\node at (12,0.5) {$\overline{1}$};
				\node at (13,0.5) {$\cdots$};
				\node at (14,0.5) {$\overline{1}$};
				
				\draw (0,-1) rectangle ++(4,1);
				\node at (0.5,-0.5) {2};
				\node at (1.5,-0.5) {$\cdots$};
				\node at (2.5,-0.5) {2};
				
				\draw (3,-1) rectangle ++(4,1);
				\node at (3.5,-0.5) {0};
				\node at (4.5,-0.5) {$\overline{1}$};
				\node at (5.5,-0.5) {$\cdots$};
				\node at (6.5,-0.5) {$\overline{1}$};
			\end{tikzpicture}
			\caption{} 
			\label{OSP-KN5}
		\end{center}
	\end{figure}
	\begin{figure}[H]
		\begin{center}
			\begin{tikzpicture}[scale=0.7, every node/.style={font=\large}]
				
				\draw (0,0) rectangle ++(4,1);
				\node at (0.5,0.5) {1};
				\node at (1.5,0.5) {$\cdots$};
				\node at (2.5,0.5) {1};
				
				\draw (3,0) rectangle ++(1,1);
				\node at (3.5,0.5) {0};
				
				\draw (4,0) rectangle ++(10,1);
				\node at (5,0.5) {$\overline{2}$};
				\node at (6.5,0.5) {$\cdots$};
				\node at (8,0.5) {$\overline{2}$};
				\node at (9.5,0.5) {$\overline{1}$};
				\node at (11,0.5) {$\cdots$};
				\node at (12.5,0.5) {$\overline{1}$};
				
				\draw (0,-1) rectangle ++(4,1);
				\node at (0.5,-0.5) {2};
				\node at (1.5,-0.5) {$\cdots$};
				\node at (2.5,-0.5) {2};
				
				\draw (3,-1) rectangle ++(4,1);
				\node at (3.5,-0.5) {0};
				\node at (4.5,-0.5) {$\overline{1}$};
				\node at (5.5,-0.5) {$\cdots$};
				\node at (6.5,-0.5) {$\overline{1}$};
			\end{tikzpicture}
			\caption{} 
			\label{OSP-KN6}
		\end{center}
	\end{figure}
	
	\begin{remark}
		In these tableaux, the column containing $0$ may be absent.
	\end{remark}
	
	If $T$ is as shown in Figure~\ref{OSP-KN1}, then
	\[
	0 \leq b_3 \leq b_2 + m_1 \leq 2b_2.
	\]
	
	If the entry $0$ appears in the first row of Figure~\ref{OSP-KN1}, then
	\[
	0 \leq b_4 \leq \min\left\{m_1, \frac{1}{2}(b_3 - 1)\right\} = \min\left\{m_1, \left\lfloor \frac{1}{2}b_3 \right\rfloor \right\},
	\]
	where $\lfloor \cdot \rfloor$ denotes the floor function.
	
	If the entry $0$ does not appear in the first row of Figure~\ref{OSP-KN1}, then
	\[
	0 \leq b_4 \leq \min\left\{m_1, \frac{1}{2}b_3\right\}.
	\]
	
	If $T$ falls into other cases, a similar analysis can be carried out. Moreover, since $b_4 \in \mathbb{Z}_{\geq 0}$, we always have
	\[
	0 \leq b_3 \leq \min\{b_2 + m_1,\, 2b_2\}, \quad
	0 \leq b_4 \leq \min\left\{m_1,\, \frac{1}{2}b_3\right\}.
	\]
	
	Therefore, for any $T \in \mathrm{KN}_{\lambda}(4|1)$, we have
	\begin{equation} \label{不等式osp}
		\begin{aligned}
			0 &\leq b_1 \leq 2m_2, \\
			0 &\leq b_2 \leq m_1 + b_1, \\
			0 &\leq b_3 \leq \min\{b_2 + m_1,\, 2b_2\}, \\
			0 &\leq b_4 \leq \min\{m_1,\, \tfrac{1}{2}b_3\}.
		\end{aligned}
	\end{equation}
	
	\section{Verma Vectors of Irreducible Representations of $\mathfrak{spo}(4|1)$}
	In this section, we directly apply the inequalities obtained in the previous section to construct Verma vectors of the finite dimensional irreducible representations of $\mathfrak{spo}(4|1)$. Furthermore, we establish a weight-preserving bijection between the Verma vectors and the KN tableaux of $\mathfrak{spo}(4|1)$.
	
	Let $L(\lambda)$ be the finite dimensional irreducible representation of $\mathfrak{spo}(4|1)$ with highest weight $\lambda = (m_1 + m_2)\epsilon_1 + m_2\epsilon_2$, where $m_1, m_2 \in \mathbb{Z}_{\geq 0}$. Let $v_\lambda$ be a nonzero highest weight vector in $L(\lambda)$.
	
	\begin{definition}
		The weight vector
		\begin{equation}\label{Verma vector}
			f_1^{b_4} f_2^{b_3} f_1^{b_2} f_2^{b_1} v_\lambda
		\end{equation}
		is called a Verma vector of the irreducible representation $L(\lambda)$ of $\mathfrak{spo}(4|1)$, where the parameters satisfy the following constraints:
		\begin{equation} \label{verma inequalities}
			\begin{aligned}
				0 &\leq b_1 \leq 2m_2, \\
				0 &\leq b_2 \leq m_1 + b_1, \\
				0 &\leq b_3 \leq \min\{b_2 + m_1,\, 2b_2\}, \\
				0 &\leq b_4 \leq \min\{m_1,\, \tfrac{1}{2}b_3\}.
			\end{aligned}
		\end{equation}
		Here, $f_{1}=E_{21}-E_{34}$ is the root vector corresponding to the negative simple root $\epsilon_2 - \epsilon_1$ of $\mathfrak{spo}(4|1)$, and $f_{2}=E_{52}-E_{45}$ is the root vector corresponding to the negative simple root $-\epsilon_2$.
	\end{definition}
	
	We denote the Verma vector $f_1^{b_4} f_2^{b_3} f_1^{b_2} f_2^{b_1} v_{\lambda}$ simply by $\bm{f^{b}} v_{\lambda}$, where $\bm{b} = (b_1, b_2, b_3, b_4)$. Let $H$ denote the set of all Verma vectors of $L(\lambda)$. Then we have the following theorem.
	
	\begin{theorem}\label{KN-correspondence}
		There exists a one-to-one correspondence between the set $H$ of Verma vectors of $L(\lambda)$ and the set $\mathrm{KN}_{\lambda}(4|1)$ of KN tableaux of $\mathfrak{spo}(4|1)$.
	\end{theorem}
	
	\begin{proof}
		Construct a mapping
		\begin{equation*}
			\begin{aligned}
				\psi \colon \mathrm{KN}_{\lambda}(4|1) &\longrightarrow H \\
				T &\longmapsto \bm{f^{b}}v_{\lambda}, \quad \forall T \in \mathrm{KN}_{\lambda}(4|1),
			\end{aligned}
		\end{equation*}
		where $\bm{b}=(b_{1}, b_{2}, b_{3}, b_{4})$, such that: $b_1$ is the sum of twice the number of entries in the second row of $T$ that are greater than or equal to $\overline{2}$ and the number of entries in the second row equal to $0$; $b_2$ is the sum of the number of entries in the first row strictly greater than $1$ and the number of entries in the second row strictly greater than $\overline{2}$; $b_3$ is the sum of twice the number of entries in the first row greater than or equal to $\overline{2}$ and the number of entries in the first row equal to $0$; and $b_4$ is the number of entries in the first row strictly greater than $\overline{2}$. Consequently, $\bm{b}=(b_{1}, b_{2}, b_{3}, b_{4})$ satisfies the system of inequalities \eqref{verma inequalities}, which implies $\bm{f^{b}}v_{\lambda} \in H$. Thus, $\psi$ is well-defined.
		
		For any $\bm{f^{b}}v_{\lambda} \in H$ with $\bm{b}=(b_{1}, b_{2}, b_{3}, b_{4})$, we will show that there exists a unique $T(\bm{b}) \in \mathrm{KN}_{\lambda}(4|1)$ such that $\psi(T(\bm{b})) = \bm{f^{b}}v_{\lambda}$. We proceed by discussing the cases for $\bm{b}=(b_{1}, b_{2}, b_{3}, b_{4})$:
		
		\noindent \textbf{Case (i)} $b_2 \geq m_1$, and $\frac{1}{2}(b_2 - m_1) \in \mathbb{Z}_{\geq 0}$:
		
		If $b_1$ and $b_3$ are even, the unique $T(\bm{b})$ such that $\psi(T(\bm{b})) = \bm{f^{b}}v_{\lambda}$ is shown in Figure~\ref{OSPKN1}. We use this case as an example to illustrate the uniqueness; the other cases can be demonstrated similarly. Since the sum of twice the number of entries in the second row of $T(\bm{b})$ that are greater than or equal to $\overline{2}$ and the number of entries in the second row equal to $0$ must be $b_1$, and $b_1$ is even, there are no $0$'s in the second row of $T(\bm{b})$. Consequently, the number of $2$'s in the second row of $T(\bm{b})$ must be $m_2 - \frac{1}{2}b_1$. Since the sum of the number of entries in the first row strictly greater than $1$ and the number of entries in the second row strictly greater than $\overline{2}$ must be $b_2$, with $b_2 \geq m_1$ and $\frac{1}{2}(b_2 - m_1) \in \mathbb{Z}_{\geq 0}$, the number of $\overline{1}$'s in the second row of $T(\bm{b})$ must be $\frac{1}{2}(b_2 - m_1)$. It follows that the number of $\overline{2}$'s in the second row of $T(\bm{b})$ must be $\frac{1}{2}b_1 - \frac{1}{2}(b_2 - m_1)$. Because the number of entries in the first row strictly greater than $\overline{2}$ must be $b_4$, the number of $\overline{1}$'s in the first row of $T(\bm{b})$ must be $b_4$. Furthermore, since the sum of twice the number of entries in the first row greater than or equal to $\overline{2}$ and the number of entries in the first row equal to $0$ must be $b_3$, and $b_3$ is even, there are no $0$'s in the first row of $T(\bm{b})$. Thus, the number of $\overline{2}$'s in the first row of $T(\bm{b})$ must be $\frac{1}{2}b_3 - b_4$. From this, we deduce that the number of $2$'s in the first row of $T(\bm{b})$ must be $\frac{1}{2}(b_2 + m_1) - \frac{1}{2}b_3$, and the number of $1$'s in the first row of $T(\bm{b})$ must be $m_2 - \frac{1}{2}(b_2 - m_1)$. Therefore, there is exactly one KN tableau $T(\bm{b})$ of $\mathfrak{spo}(4|1)$ satisfying $\psi(T(\bm{b})) = \bm{f^{b}}v_{\lambda}$, as shown in Figure~\ref{OSPKN1}.
		\begin{figure}[H]
			\begin{center}
				\begin{tikzpicture}[scale=0.7,
					every node/.style={font=\large},
					brace/.style={decorate,decoration={brace,  amplitude=5pt,raise=2pt}},
					brace mirror/.style={  decorate,decoration={brace,  mirror, amplitude=5pt,raise=2pt}},
					label/.style={midway,font=\scriptsize,outer sep=6pt} ]
					
					\draw (0,0) rectangle (3,1);
					\foreach \x in {0.5,2.5} \node at (\x,0.5) {1};
					\node at (1.5,0.5) {$\cdots$};
					
					\draw (3,0) rectangle (7,1);
					\foreach \x in {3.5,5.5,6.5} \node at (\x,0.5) {1};
					\node at (4.5,0.5) {$\cdots$};
					
					\draw[brace] (0,1.1) -- node[label,above] {$m_{2}-\frac{1}{2}(b_{2}-m_{1})$} (7,1.1);
					
					\draw (7,0) rectangle (16,1);
					\foreach \x in {7.5,9.5} \node at (\x,0.5) {2};
					\node at (8.5,0.5) {$\cdots$};
					\draw[brace] (7,1.1) -- node[label,above,xshift=-1mm] {$\frac{1}{2}(b_{2}+m_{1})-\frac{1}{2}b_3$} (9.6,1.1);
					
					\foreach \x in {10.5,12.5} \node at (\x,0.5) {$\overline{2}$};
					\node at (11.5,0.5) {$\cdots$};
					\draw[brace] (10.3,1.1) -- node[label,above] {$\frac{1}{2}b_{3}-b_{4}$} (12.6,1.1);
					
					\foreach \x in {13.5,15.5} \node at (\x,0.5) {$\overline{1}$};
					\node at (14.5,0.5) {$\cdots$};
					\draw[brace] (13.3,1.1) -- node[label,above] {$b_{4}$} (16,1.1);

					\draw (0,-1) rectangle (3,0);
					\foreach \x in {0.5,2.5} \node at (\x,-0.5) {2};
					\node at (1.5,-0.5) {$\cdots$};
					\draw[brace mirror] (0,-1.1) -- node[label,below] {$m_{2}-\frac{1}{2}b_{1}$} (3,-1.1);
					
					\draw (3,-1) rectangle (7,0);
					
					\foreach \x in {3.5,5.5,6.5} \node at (\x,-0.5) {$\overline{2}$};
					\node at (4.5,-0.5) {$\cdots$};
					\draw[brace mirror] (3,-1.1) -- node[label,below] {$\frac{1}{2}b_{1}-\frac{1}{2}(b_{2}-m_{1})$} (7,-1.1);
					
					\draw (7,-1) rectangle (12,0);
					\foreach \x in {8,11} \node at (\x,-0.5) {$\overline{1}$};
					\node at (9.5,-0.5) {$\cdots$};
					\draw[brace mirror] (7,-1.1) -- node[label,below] {$\frac{1}{2}(b_{2}-m_{1})$} (12,-1.1);
					
				\end{tikzpicture}
				\caption{}
				\label{OSPKN1}
			\end{center}
		\end{figure}
		
		If $b_1$ is even and $b_3$ is odd, the unique $T(\bm{b})$ such that $\psi(T(\bm{b})) = \bm{f^{b}}v_{\lambda}$ is shown in Figure~\ref{OSPKN2}. 
		\begin{figure}[H]
			\begin{center}
				\begin{tikzpicture}[scale=0.7,
					every node/.style={font=\large},
					brace/.style={decorate,decoration={brace,  amplitude=5pt,raise=2pt}},
					brace mirror/.style={  decorate,decoration={brace,  mirror, amplitude=5pt,raise=2pt}},
					label/.style={midway,font=\scriptsize,outer sep=6pt} ]
					
					\draw (0,0) rectangle (3,1);
					\foreach \x in {0.5,2.5} \node at (\x,0.5) {1};
					\node at (1.5,0.5) {$\cdots$};
					
					\draw (3,0) rectangle (7,1);
					\foreach \x in {3.5,5.5,6.5} \node at (\x,0.5) {1};
					\node at (4.5,0.5) {$\cdots$};
					
					\draw[brace] (0,1.1) -- node[label,above,xshift=-1mm] {$m_{2}-\frac{1}{2}(b_{2}-m_{1})$} (7,1.1);
					
					\draw (7,0) rectangle (17,1);
					\foreach \x in {7.5,9.5} \node at (\x,0.5) {2};
					\node at (8.5,0.5) {$\cdots$};
					\draw[brace] (7,1.1) -- node[label,above] {$\frac{1}{2}(b_{2}+m_{1})-\frac{1}{2}(b_{3}-1)-1$} (9.6,1.1);
					
					\node at (10.5,0.5) {0};
					
					\foreach \x in {11.5,13.5} \node at (\x,0.5) {$\overline{2}$};
					\node at (12.5,0.5) {$\cdots$};
					\draw[brace] (11.3,1.1) -- node[label,above,xshift=3mm] {$\frac{1}{2}(b_{3}-1)-b_{4}$} (13.6,1.1);
					
					\foreach \x in {14.5,16.5} \node at (\x,0.5) {$\overline{1}$};
					\node at (15.5,0.5) {$\cdots$};
					\draw[brace] (14.3,1.1) -- node[label,above] {$b_{4}$} (17,1.1);

					\draw (0,-1) rectangle (3,0);
					\foreach \x in {0.5,2.5} \node at (\x,-0.5) {2};
					\node at (1.5,-0.5) {$\cdots$};
					\draw[brace mirror] (0,-1.1) -- node[label,below] {$m_{2}-\frac{1}{2}b_{1}$} (3,-1.1);
					
					\draw (3,-1) rectangle (7,0);
					
					\foreach \x in {3.5,5.5,6.5} \node at (\x,-0.5) {$\overline{2}$};
					\node at (4.5,-0.5) {$\cdots$};
					\draw[brace mirror] (3,-1.1) -- node[label,below] {$\frac{1}{2}b_{1}-\frac{1}{2}(b_{2}-m_{1})$} (7,-1.1);
					
					\draw (7,-1) rectangle (12,0);
					\foreach \x in {8,11} \node at (\x,-0.5) {$\overline{1}$};
					\node at (9.5,-0.5) {$\cdots$};
					\draw[brace mirror] (7,-1.1) -- node[label,below] {$\frac{1}{2}(b_{2}-m_{1})$} (12,-1.1);
					
				\end{tikzpicture}
				\caption{}
				\label{OSPKN2}
			\end{center}
		\end{figure}
		
		If $b_1$ and $b_3$ are odd, the unique $T(\bm{b})$ such that $\psi(T(\bm{b})) = \bm{f^{b}}v_{\lambda}$ is shown in Figure~\ref{OSPKN3}. 
		\begin{figure}[H]
			\begin{center}
				\begin{tikzpicture}[scale=0.7,
					every node/.style={font=\large},
					brace/.style={decorate,decoration={brace,  amplitude=5pt,raise=2pt}},
					brace mirror/.style={  decorate,decoration={brace,  mirror, amplitude=5pt,raise=2pt}},
					label/.style={midway,font=\scriptsize,outer sep=6pt} ]
					
					\draw (0,0) rectangle (4,1);
					\foreach \x in {0.5,2.5} \node at (\x,0.5) {1};
					\node at (1.5,0.5) {$\cdots$};
					
					\draw (3,0) rectangle (8,1);
					\foreach \x in {3.5,4.5,6.5,7.5} \node at (\x,0.5) {1};
					\node at (5.5,0.5) {$\cdots$};
					
					\draw[brace] (0,1.1) -- node[label,above] {$m_{2}-\frac{1}{2}(b_{2}-m_{1})$} (8,1.1);
					
					\draw (8,0) rectangle (18,1);
					\foreach \x in {8.5,10.5} \node at (\x,0.5) {2};
					\node at (9.5,0.5) {$\cdots$};
					\draw[brace] (8,1.1) -- node[label,above,xshift=-2mm] {$\frac{1}{2}(b_{2}+m_{1})-\frac{1}{2}(b_{3}-1)-1$} (10.6,1.1);
					
					\node at (11.5,0.5) {0};
					
					\foreach \x in {12.5,14.5} \node at (\x,0.5) {$\overline{2}$};
					\node at (13.5,0.5) {$\cdots$};
					\draw[brace] (12.3,1.1) -- node[label,above,xshift=2mm] {$\frac{1}{2}(b_{3}-1)-b_{4}$} (14.6,1.1);
					
					\foreach \x in {15.5,17.5} \node at (\x,0.5) {$\overline{1}$};
					\node at (16.5,0.5) {$\cdots$};
					\draw[brace] (15.3,1.1) -- node[label,above] {$b_{4}$} (18,1.1);
					
					\draw (0,-1) rectangle (4,0);
					\foreach \x in {0.5,2.5} \node at (\x,-0.5) {2};
					\node at (1.5,-0.5) {$\cdots$};
					\draw[brace mirror] (0,-1.1) -- node[label,below,xshift=-2mm] {$m_{2}-\frac{1}{2}(b_{1}-1)-1$} (3,-1.1);
					
					\draw (3,-1) rectangle (8,0);
					\node at (3.5,-0.5) {0};
					\foreach \x in {4.5,6.5,7.5} \node at (\x,-0.5) {$\overline{2}$};
					\node at (5.5,-0.5) {$\cdots$};
					\draw[brace mirror] (4,-1.1) -- node[label,below,xshift=1mm] {$\frac{1}{2}(b_{1}-1)-\frac{1}{2}(b_{2}-m_{1})$} (8,-1.1);
					
					\draw (8,-1) rectangle (13,0);
					\foreach \x in {9,12} \node at (\x,-0.5) {$\overline{1}$};
					\node at (10.5,-0.5) {$\cdots$};
					\draw[brace mirror] (8,-1.1) -- node[label,below] {$\frac{1}{2}(b_{2}-m_{1})$} (13,-1.1);
					
				\end{tikzpicture}
				\caption{}
				\label{OSPKN3}
			\end{center}
		\end{figure}
		
		If $b_1$ is odd and $b_3$ is even, the unique $T(\bm{b})$ such that $\psi(T(\bm{b})) = \bm{f^{b}}v_{\lambda}$ is shown in Figure~\ref{OSPKN4}.
		\begin{figure}[H]
			\begin{center}
				\begin{tikzpicture}[scale=0.7,
					every node/.style={font=\large},
					brace/.style={decorate,decoration={brace,  amplitude=5pt,raise=2pt}},
					brace mirror/.style={  decorate,decoration={brace,  mirror, amplitude=5pt,raise=2pt}},
					label/.style={midway,font=\scriptsize,outer sep=6pt} ]
					
					\draw (0,0) rectangle (4,1);
					\foreach \x in {0.5,2.5} \node at (\x,0.5) {1};
					\node at (1.5,0.5) {$\cdots$};
					
					\draw (3,0) rectangle (8,1);
					\foreach \x in {3.5,4.5,6.5,7.5} \node at (\x,0.5) {1};
					\node at (5.5,0.5) {$\cdots$};
					
					\draw[brace] (0,1.1) -- node[label,above] {$m_{2}-\frac{1}{2}(b_{2}-m_{1})$} (8,1.1);
					
					\draw (8,0) rectangle (17,1);
					\foreach \x in {8.5,10.5} \node at (\x,0.5) {2};
					\node at (9.5,0.5) {$\cdots$};
					\draw[brace] (8,1.1) -- node[label,above,xshift=-1mm] {$\frac{1}{2}(b_{2}+m_{1})-\frac{1}{2}b_3$} (10.6,1.1);
					
					\foreach \x in {11.5,13.5} \node at (\x,0.5) {$\overline{2}$};
					\node at (12.5,0.5) {$\cdots$};
					\draw[brace] (11.3,1.1) -- node[label,above] {$\frac{1}{2}b_{3}-b_{4}$} (13.6,1.1);
					
					\foreach \x in {14.5,16.5} \node at (\x,0.5) {$\overline{1}$};
					\node at (15.5,0.5) {$\cdots$};
					\draw[brace] (14.3,1.1) -- node[label,above] {$b_{4}$} (17,1.1);
					
					\draw (0,-1) rectangle (4,0);
					\foreach \x in {0.5,2.5} \node at (\x,-0.5) {2};
					\node at (1.5,-0.5) {$\cdots$};
					\draw[brace mirror] (0,-1.1) -- node[label,below,xshift=-2mm] {$m_{2}-\frac{1}{2}(b_{1}-1)-1$} (3,-1.1);
					
					\draw (3,-1) rectangle (8,0);
					\node at (3.5,-0.5) {0};
					\foreach \x in {4.5,6.5,7.5} \node at (\x,-0.5) {$\overline{2}$};
					\node at (5.5,-0.5) {$\cdots$};
					\draw[brace mirror] (4,-1.1) -- node[label,below] {$\frac{1}{2}(b_{1}-1)-\frac{1}{2}(b_{2}-m_{1})$} (8,-1.1);
					
					\draw (8,-1) rectangle (13,0);
					\foreach \x in {9,12} \node at (\x,-0.5) {$\overline{1}$};
					\node at (10.5,-0.5) {$\cdots$};
					\draw[brace mirror] (8,-1.1) -- node[label,below] {$\frac{1}{2}(b_{2}-m_{1})$} (13,-1.1);
					
				\end{tikzpicture}
				\caption{}
				\label{OSPKN4}
			\end{center}
		\end{figure}
		
		\noindent \textbf{Case (ii)} $b_2 \geq m_1$, and $\frac{1}{2}(b_2 - m_1 - 1) \in \mathbb{Z}_{\geq 0}$:
		
		If $b_1$ and $b_3$ are even, the unique $T(\bm{b})$ such that $\psi(T(\bm{b})) = \bm{f^{b}}v_{\lambda}$ is shown in Figure~\ref{OSPKN5}. 
		\begin{figure}[H]
			\begin{center}
				\begin{tikzpicture}[scale=0.7,
					every node/.style={font=\large},
					brace/.style={decorate,decoration={brace,  amplitude=5pt,raise=2pt}},
					brace mirror/.style={  decorate,decoration={brace,  mirror, amplitude=5pt,raise=2pt}},
					label/.style={midway,font=\scriptsize,outer sep=6pt} ]
					
					\draw (0,0) rectangle (3,1);
					\foreach \x in {0.5,2.5} \node at (\x,0.5) {1};
					\node at (1.5,0.5) {$\cdots$};
					
					\draw (3,0) rectangle (7,1);
					\foreach \x in {3.5,5.5} \node at (\x,0.5) {1};
					\node at (4.5,0.5) {$\cdots$};
					\node at (6.5,0.5) {2};
					
					\draw[brace] (0,1.1) -- node[label,above] {$m_{2}-\frac{1}{2}(b_{2}-m_{1}+1)$} (5.5,1.1);
					
					\draw (7,0) rectangle (16,1);
					\foreach \x in {7.5,9.5} \node at (\x,0.5) {2};
					\node at (8.5,0.5) {$\cdots$};
					\draw[brace] (6.4,1.1) -- node[label,above,xshift=-1mm] {$\frac{1}{2}(b_{2}+m_{1}+1)-\frac{1}{2}b_3$} (9.6,1.1);
					
					\foreach \x in {10.5,12.5} \node at (\x,0.5) {$\overline{2}$};
					\node at (11.5,0.5) {$\cdots$};
					\draw[brace] (10.3,1.1) -- node[label,above] {$\frac{1}{2}b_{3}-b_{4}$} (12.6,1.1);
					
					\foreach \x in {13.5,15.5} \node at (\x,0.5) {$\overline{1}$};
					\node at (14.5,0.5) {$\cdots$};
					\draw[brace] (13.3,1.1) -- node[label,above] {$b_{4}$} (16,1.1);

					\draw (0,-1) rectangle (3,0);
					\foreach \x in {0.5,2.5} \node at (\x,-0.5) {2};
					\node at (1.5,-0.5) {$\cdots$};
					\draw[brace mirror] (0,-1.1) -- node[label,below] {$m_{2}-\frac{1}{2}b_{1}$} (3,-1.1);
					
					\draw (3,-1) rectangle (7,0);
					
					\foreach \x in {3.5,5.5,6.5} \node at (\x,-0.5) {$\overline{2}$};
					\node at (4.5,-0.5) {$\cdots$};
					\draw[brace mirror] (3,-1.1) -- node[label,below] {$\frac{1}{2}b_{1}-\frac{1}{2}(b_{2}-m_{1}-1)$} (7,-1.1);
					
					\draw (7,-1) rectangle (12,0);
					\foreach \x in {8,11} \node at (\x,-0.5) {$\overline{1}$};
					\node at (9.5,-0.5) {$\cdots$};
					\draw[brace mirror] (7,-1.1) -- node[label,below] {$\frac{1}{2}(b_{2}-m_{1}-1)$} (12,-1.1);
					
				\end{tikzpicture}
				\caption{}
				\label{OSPKN5}
			\end{center}
		\end{figure}
		
		If $b_1$ is even, $b_3$ is odd, and $b_3 \neq b_2 + m_1$, the unique $T(\bm{b})$ such that $\psi(T(\bm{b})) = \bm{f^{b}}v_{\lambda}$ is shown in Figure~\ref{OSPKN6}.
		\begin{figure}[H]
			\begin{center}
				\begin{tikzpicture}[scale=0.75,
					every node/.style={font=\large},
					brace/.style={decorate,decoration={brace,  amplitude=5pt,raise=2pt}},
					brace mirror/.style={  decorate,decoration={brace,  mirror, amplitude=5pt,raise=2pt}},
					label/.style={midway,font=\scriptsize,outer sep=6pt} ]
					
					\draw (0,0) rectangle (3,1);
					\foreach \x in {0.5,2.5} \node at (\x,0.5) {1};
					\node at (1.5,0.5) {$\cdots$};
					
					\draw (3,0) rectangle (7,1);
					\foreach \x in {3.5,5.5} \node at (\x,0.5) {1};
					\node at (4.5,0.5) {$\cdots$};
					\node at (6.5,0.5) {2};
					
					\draw[brace] (0,1.1) -- node[label,above,xshift=-2mm] {$m_{2}-\frac{1}{2}(b_{2}-m_{1}+1)$} (5.5,1.1);
					
					\draw (7,0) rectangle (17,1);
					\foreach \x in {7.5,9.5} \node at (\x,0.5) {2};
					\node at (8.5,0.5) {$\cdots$};
					\draw[brace] (6.4,1.1) -- node[label,above] {$\frac{1}{2}(b_{2}+m_{1}+1)-\frac{1}{2}(b_{3}-1)-1$} (9.6,1.1);
					
					\node at (10.5,0.5) {0};
					
					\foreach \x in {11.5,13.5} \node at (\x,0.5) {$\overline{2}$};
					\node at (12.5,0.5) {$\cdots$};
					\draw[brace] (11.3,1.1) -- node[label,above,xshift=2mm] {$\frac{1}{2}(b_{3}-1)-b_{4}$} (13.6,1.1);
					
					\foreach \x in {14.5,16.5} \node at (\x,0.5) {$\overline{1}$};
					\node at (15.5,0.5) {$\cdots$};
					\draw[brace] (14.3,1.1) -- node[label,above] {$b_{4}$} (17,1.1);

					\draw (0,-1) rectangle (3,0);
					\foreach \x in {0.5,2.5} \node at (\x,-0.5) {2};
					\node at (1.5,-0.5) {$\cdots$};
					\draw[brace mirror] (0,-1.1) -- node[label,below] {$m_{2}-\frac{1}{2}b_{1}$} (3,-1.1);
					
					\draw (3,-1) rectangle (7,0);
					
					\foreach \x in {3.5,5.5,6.5} \node at (\x,-0.5) {$\overline{2}$};
					\node at (4.5,-0.5) {$\cdots$};
					\draw[brace mirror] (3,-1.1) -- node[label,below] {$\frac{1}{2}b_{1}-\frac{1}{2}(b_{2}-m_{1}-1)$} (7,-1.1);
					
					\draw (7,-1) rectangle (12,0);
					\foreach \x in {8,11} \node at (\x,-0.5) {$\overline{1}$};
					\node at (9.5,-0.5) {$\cdots$};
					\draw[brace mirror] (7,-1.1) -- node[label,below] {$\frac{1}{2}(b_{2}-m_{1}-1)$} (12,-1.1);
					
				\end{tikzpicture}
				\caption{}
				\label{OSPKN6}
			\end{center}
		\end{figure}
		
		 If $b_1$ is even, $b_3$ is odd, and $b_3 = b_2 + m_1$, the unique $T(\bm{b})$ such that $\psi(T(\bm{b})) = \bm{f^{b}}v_{\lambda}$ is shown in Figure~\ref{OSPKN7}. 
		\begin{figure}[H]
			\begin{center}
				\begin{tikzpicture}[scale=0.7,
					every node/.style={font=\large},
					brace/.style={decorate,decoration={brace,  amplitude=5pt,raise=2pt}},
					brace mirror/.style={  decorate,decoration={brace,  mirror, amplitude=5pt,raise=2pt}},
					label/.style={midway,font=\scriptsize,outer sep=6pt} ]
					
					\draw (0,0) rectangle (3,1);
					\foreach \x in {0.5,2.5} \node at (\x,0.5) {1};
					\node at (1.5,0.5) {$\cdots$};
					
					\draw (3,0) rectangle (7,1);
					\foreach \x in {3.5,5.5} \node at (\x,0.5) {1};
					\node at (4.5,0.5) {$\cdots$};
					\node at (6.5,0.5) {0};
					
					\draw[brace] (0,1.1) -- node[label,above] {$m_{2}-\frac{1}{2}(b_{2}-m_{1}+1)$} (5.5,1.1);
					
					\draw (7,0) rectangle (17,1);
					\foreach \x in {7.9,9.9} \node at (\x,0.5) {$\overline{2}$};
					\node at (8.9,0.5) {$\cdots$};

					\foreach \x in {11.3,13.3} \node at (\x,0.5) {$\overline{2}$};
					\node at (12.3,0.5) {$\cdots$};
					\draw[brace] (7,1.1) -- node[label,above] {$\frac{1}{2}(b_{3}-1)-b_{4}$} (13.6,1.1);
					
					\foreach \x in {14.5,16.5} \node at (\x,0.5) {$\overline{1}$};
					\node at (15.5,0.5) {$\cdots$};
					\draw[brace] (14.3,1.1) -- node[label,above] {$b_{4}$} (17,1.1);

					\draw (0,-1) rectangle (3,0);
					\foreach \x in {0.5,2.5} \node at (\x,-0.5) {2};
					\node at (1.5,-0.5) {$\cdots$};
					\draw[brace mirror] (0,-1.1) -- node[label,below] {$m_{2}-\frac{1}{2}b_{1}$} (3,-1.1);
					
					\draw (3,-1) rectangle (7,0);
					
					\foreach \x in {3.5,5.5,6.5} \node at (\x,-0.5) {$\overline{2}$};
					\node at (4.5,-0.5) {$\cdots$};
					\draw[brace mirror] (3,-1.1) -- node[label,below] {$\frac{1}{2}b_{1}-\frac{1}{2}(b_{2}-m_{1}-1)$} (7,-1.1);
					
					\draw (7,-1) rectangle (12,0);
					\foreach \x in {8,11} \node at (\x,-0.5) {$\overline{1}$};
					\node at (9.5,-0.5) {$\cdots$};
					\draw[brace mirror] (7,-1.1) -- node[label,below] {$\frac{1}{2}(b_{2}-m_{1}-1)$} (12,-1.1);
					
				\end{tikzpicture}
				\caption{}
				\label{OSPKN7}
			\end{center}
		\end{figure}
		
		If $b_1$ and $b_3$ are odd, with $b_2 \neq b_1 + m_1$ and $b_3 \neq b_2 + m_1$, the unique $T(\bm{b})$ such that $\psi(T(\bm{b})) = \bm{f^{b}}v_{\lambda}$ is shown in Figure~\ref{OSPKN8}. 
		\begin{figure}[H]
			\begin{center}
				\begin{tikzpicture}[scale=0.75,
					every node/.style={font=\large},
					brace/.style={decorate,decoration={brace,  amplitude=5pt,raise=2pt}},
					brace mirror/.style={  decorate,decoration={brace,  mirror, amplitude=5pt,raise=2pt}},
					label/.style={midway,font=\scriptsize,outer sep=6pt} ]
					
					\draw (0,0) rectangle (4,1);
					\foreach \x in {0.5,2.5} \node at (\x,0.5) {1};
					\node at (1.5,0.5) {$\cdots$};
					
					\draw (3,0) rectangle (8,1);
					\foreach \x in {3.5,4.5,6.5} \node at (\x,0.5) {1};
					\node at (5.5,0.5) {$\cdots$};
					\node at (7.5,0.5) {2};
					
					\draw[brace] (0,1.1) -- node[label,above] {$m_{2}-\frac{1}{2}(b_{2}-m_{1}+1)$} (6.5,1.1);
					
					\draw (8,0) rectangle (18,1);
					\foreach \x in {8.5,10.5} \node at (\x,0.5) {2};
					\node at (9.5,0.5) {$\cdots$};
					\draw[brace] (7.4,1.1) -- node[label,above,xshift=-1mm] {$\frac{1}{2}(b_{2}+m_{1}+1)-\frac{1}{2}(b_{3}-1)-1$} (10.6,1.1);
					
					\node at (11.5,0.5) {0};
					
					\foreach \x in {12.5,14.5} \node at (\x,0.5) {$\overline{2}$};
					\node at (13.5,0.5) {$\cdots$};
					\draw[brace] (12.3,1.1) -- node[label,above,xshift=2mm] {$\frac{1}{2}(b_{3}-1)-b_{4}$} (14.6,1.1);
					
					\foreach \x in {15.5,17.5} \node at (\x,0.5) {$\overline{1}$};
					\node at (16.5,0.5) {$\cdots$};
					\draw[brace] (15.3,1.1) -- node[label,above] {$b_{4}$} (18,1.1);
					
					\draw (0,-1) rectangle (4,0);
					\foreach \x in {0.5,2.5} \node at (\x,-0.5) {2};
					\node at (1.5,-0.5) {$\cdots$};
					\draw[brace mirror] (0,-1.1) -- node[label,below,xshift=-3mm] {$m_{2}-\frac{1}{2}(b_{1}-1)-1$} (3,-1.1);
					
					\draw (3,-1) rectangle (8,0);
					\node at (3.5,-0.5) {0};
					\foreach \x in {4.5,6.5,7.5} \node at (\x,-0.5) {$\overline{2}$};
					\node at (5.5,-0.5) {$\cdots$};
					\draw[brace mirror] (4,-1.1) -- node[label,below] {$\frac{1}{2}(b_{1}-1)-\frac{1}{2}(b_{2}-m_{1}-1)$} (8,-1.1);
					
					\draw (8,-1) rectangle (13,0);
					\foreach \x in {9,12} \node at (\x,-0.5) {$\overline{1}$};
					\node at (10.5,-0.5) {$\cdots$};
					\draw[brace mirror] (8,-1.1) -- node[label,below] {$\frac{1}{2}(b_{2}-m_{1}-1)$} (13,-1.1);
					
				\end{tikzpicture}
				\caption{}
				\label{OSPKN8}
			\end{center}
		\end{figure}
		
		If $b_1$ and $b_3$ are odd, with $b_2 \neq b_1 + m_1$ and $b_3 = b_2 + m_1$, the unique $T(\bm{b})$ such that $\psi(T(\bm{b})) = \bm{f^{b}}v_{\lambda}$ is shown in Figure~\ref{OSPKN9}. 
		\begin{figure}[H]
			\begin{center}
				\begin{tikzpicture}[scale=0.75,
					every node/.style={font=\large},
					brace/.style={decorate,decoration={brace,  amplitude=5pt,raise=2pt}},
					brace mirror/.style={  decorate,decoration={brace,  mirror, amplitude=5pt,raise=2pt}},
					label/.style={midway,font=\scriptsize,outer sep=6pt} ]
					
					\draw (0,0) rectangle (4,1);
					\foreach \x in {0.5,2.5} \node at (\x,0.5) {1};
					\node at (1.5,0.5) {$\cdots$};
					
					\draw (3,0) rectangle (8,1);
					\foreach \x in {3.5,4.5,6.5} \node at (\x,0.5) {1};
					\node at (5.5,0.5) {$\cdots$};
					\node at (7.5,0.5) {0};
					
					\draw[brace] (0,1.1) -- node[label,above] {$m_{2}-\frac{1}{2}(b_{2}-m_{1}+1)$} (6.5,1.1);
					
					\draw (8,0) rectangle (18,1);
					\foreach \x in {8.9,10.9} \node at (\x,0.5) {$\overline{2}$};
					\node at (9.9,0.5) {$\cdots$};

					\foreach \x in {12.3,14.3} \node at (\x,0.5) {$\overline{2}$};
					\node at (13.3,0.5) {$\cdots$};
					\draw[brace] (8,1.1) -- node[label,above] {$\frac{1}{2}(b_{3}-1)-b_{4}$} (14.6,1.1);
					
					\foreach \x in {15.5,17.5} \node at (\x,0.5) {$\overline{1}$};
					\node at (16.5,0.5) {$\cdots$};
					\draw[brace] (15.3,1.1) -- node[label,above] {$b_{4}$} (18,1.1);
					
					\draw (0,-1) rectangle (4,0);
					\foreach \x in {0.5,2.5} \node at (\x,-0.5) {2};
					\node at (1.5,-0.5) {$\cdots$};
					\draw[brace mirror] (0,-1.1) -- node[label,below,xshift=-3mm] {$m_{2}-\frac{1}{2}(b_{1}-1)-1$} (3,-1.1);
					
					\draw (3,-1) rectangle (8,0);
					\node at (3.5,-0.5) {0};
					\foreach \x in {4.5,6.5,7.5} \node at (\x,-0.5) {$\overline{2}$};
					\node at (5.5,-0.5) {$\cdots$};
					\draw[brace mirror] (4,-1.1) -- node[label,below] {$\frac{1}{2}(b_{1}-1)-\frac{1}{2}(b_{2}-m_{1}-1)$} (8,-1.1);
					
					\draw (8,-1) rectangle (13,0);
					\foreach \x in {9,12} \node at (\x,-0.5) {$\overline{1}$};
					\node at (10.5,-0.5) {$\cdots$};
					\draw[brace mirror] (8,-1.1) -- node[label,below] {$\frac{1}{2}(b_{2}-m_{1}-1)$} (13,-1.1);
					
				\end{tikzpicture}
				\caption{}
				\label{OSPKN9}
			\end{center}
		\end{figure}
		
		If $b_1$ and $b_3$ are odd, with $b_2 = b_1 + m_1$ and $b_3 \neq b_2 + m_1$, the unique $T(\bm{b})$ such that $\psi(T(\bm{b})) = \bm{f^{b}}v_{\lambda}$ is shown in Figure~\ref{OSPKN10}. 
		\begin{figure}[H]
			\begin{center}
				\begin{tikzpicture}[scale=0.85,
					every node/.style={font=\large},
					brace/.style={decorate,decoration={brace,  amplitude=5pt,raise=2pt}},
					brace mirror/.style={  decorate,decoration={brace,  mirror, amplitude=5pt,raise=2pt}},
					label/.style={midway,font=\scriptsize,outer sep=6pt} ]
					
					\draw (0,0) rectangle (4,1);
					\foreach \x in {0.5,2.5} \node at (\x,0.5) {1};
					\node at (1.5,0.5) {$\cdots$};

					\node at (3.5,0.5) {2};

					\draw[brace] (0,1.1) -- node[label,above,xshift=-8mm] {$m_{2}-\frac{1}{2}(b_{2}-m_{1}+1)$} (3,1.1);
					
					\draw (3,0) rectangle (14,1);
					\foreach \x in {4.5,6.5} \node at (\x,0.5) {2};
					\node at (5.5,0.5) {$\cdots$};
					\draw[brace] (3,1.1) -- node[label,above,xshift=6mm] {$\frac{1}{2}(b_{2}+m_{1}+1)-\frac{1}{2}(b_{3}-1)-1$} (6.6,1.1);
					
					\node at (7.5,0.5) {0};
					
					\foreach \x in {8.5,10.5} \node at (\x,0.5) {$\overline{2}$};
					\node at (9.5,0.5) {$\cdots$};
					\draw[brace] (8.3,1.1) -- node[label,above,xshift=2mm] {$\frac{1}{2}(b_{3}-1)-b_{4}$} (10.6,1.1);
					
					\foreach \x in {11.5,13.5} \node at (\x,0.5) {$\overline{1}$};
					\node at (12.5,0.5) {$\cdots$};
					\draw[brace] (11.3,1.1) -- node[label,above] {$b_{4}$} (14,1.1);
					
					\draw (0,-1) rectangle (4,0);
					\foreach \x in {0.5,2.5} \node at (\x,-0.5) {2};
					\node at (1.5,-0.5) {$\cdots$};
					\draw[brace mirror] (0,-1.1) -- node[label,below] {$m_{2}-\frac{1}{2}(b_{1}-1)-1$} (3,-1.1);

					\node at (3.5,-0.5) {0};
					
					\draw (3,-1) rectangle (9,0);
					\foreach \x in {5,8} \node at (\x,-0.5) {$\overline{1}$};
					\node at (6.5,-0.5) {$\cdots$};
					\draw[brace mirror] (4,-1.1) -- node[label,below] {$\frac{1}{2}(b_{2}-m_{1}-1)$} (9,-1.1);
					
				\end{tikzpicture}
				\caption{}
				\label{OSPKN10}
			\end{center}
		\end{figure}
		
		If $b_1$ and $b_3$ are odd, with $b_2 = b_1 + m_1$ and $b_3 = b_2 + m_1$, the unique $T(\bm{b})$ such that $\psi(T(\bm{b})) = \bm{f^{b}}v_{\lambda}$ is shown in Figure~\ref{OSPKN11}. 
		\begin{figure}[H]
			\begin{center}
				\begin{tikzpicture}[scale=0.7,
					every node/.style={font=\large},
					brace/.style={decorate,decoration={brace,  amplitude=5pt,raise=2pt}},
					brace mirror/.style={  decorate,decoration={brace,  mirror, amplitude=5pt,raise=2pt}},
					label/.style={midway,font=\scriptsize,outer sep=6pt} ]
					
					\draw (0,0) rectangle (4,1);
					\foreach \x in {0.5,2.5} \node at (\x,0.5) {1};
					\node at (1.5,0.5) {$\cdots$};
					
					\node at (3.5,0.5) {0};
					
					\draw[brace] (0,1.1) -- node[label,above] {$m_{2}-\frac{1}{2}(b_{2}-m_{1}+1)$} (3,1.1);
					
					\draw (3,0) rectangle (14,1);
					\foreach \x in {4.9,6.9} \node at (\x,0.5) {$\overline{2}$};
					\node at (5.9,0.5) {$\cdots$};
					
					\foreach \x in {8.3,10.3} \node at (\x,0.5) {$\overline{2}$};
					\node at (9.3,0.5) {$\cdots$};
					\draw[brace] (4,1.1) -- node[label,above] {$\frac{1}{2}(b_{3}-1)-b_{4}$} (10.6,1.1);
					
					\foreach \x in {11.5,13.5} \node at (\x,0.5) {$\overline{1}$};
					\node at (12.5,0.5) {$\cdots$};
					\draw[brace] (11.3,1.1) -- node[label,above] {$b_{4}$} (14,1.1);
					
					\draw (0,-1) rectangle (4,0);
					\foreach \x in {0.5,2.5} \node at (\x,-0.5) {2};
					\node at (1.5,-0.5) {$\cdots$};
					\draw[brace mirror] (0,-1.1) -- node[label,below] {$m_{2}-\frac{1}{2}(b_{1}-1)-1$} (3,-1.1);

					\node at (3.5,-0.5) {0};
					
					\draw (3,-1) rectangle (9,0);
					\foreach \x in {5,8} \node at (\x,-0.5) {$\overline{1}$};
					\node at (6.5,-0.5) {$\cdots$};
					\draw[brace mirror] (4,-1.1) -- node[label,below] {$\frac{1}{2}(b_{2}-m_{1}-1)$} (9,-1.1);
					
				\end{tikzpicture}
				\caption{}
				\label{OSPKN11}
			\end{center}
		\end{figure}
		
		If $b_1$ is odd, $b_3$ is even, and $b_2 \neq b_1 + m_1$, the unique $T(\bm{b})$ such that $\psi(T(\bm{b})) = \bm{f^{b}}v_{\lambda}$ is shown in Figure~\ref{OSPKN12}. 
		\begin{figure}[H]
			\begin{center}
				\begin{tikzpicture}[scale=0.75,
					every node/.style={font=\large},
					brace/.style={decorate,decoration={brace,  amplitude=5pt,raise=2pt}},
					brace mirror/.style={  decorate,decoration={brace,  mirror, amplitude=5pt,raise=2pt}},
					label/.style={midway,font=\scriptsize,outer sep=6pt} ]
					
					\draw (0,0) rectangle (4,1);
					\foreach \x in {0.5,2.5} \node at (\x,0.5) {1};
					\node at (1.5,0.5) {$\cdots$};
					
					\draw (3,0) rectangle (8,1);
					\foreach \x in {3.5,4.5,6.5} \node at (\x,0.5) {1};
					\node at (7.5,0.5) {2};
					\node at (5.5,0.5) {$\cdots$};
					
					\draw[brace] (0,1.1) -- node[label,above] {$m_{2}-\frac{1}{2}(b_{2}-m_{1}+1)$} (6.6,1.1);
					
					\draw (8,0) rectangle (17,1);
					\foreach \x in {8.5,10.5} \node at (\x,0.5) {2};
					\node at (9.5,0.5) {$\cdots$};
					\draw[brace] (7.4,1.1) -- node[label,above,xshift=-1mm] {$\frac{1}{2}(b_{2}+m_{1}+1)-\frac{1}{2}b_3$} (10.6,1.1);
					
					\foreach \x in {11.5,13.5} \node at (\x,0.5) {$\overline{2}$};
					\node at (12.5,0.5) {$\cdots$};
					\draw[brace] (11.3,1.1) -- node[label,above] {$\frac{1}{2}b_{3}-b_{4}$} (13.6,1.1);
					
					\foreach \x in {14.5,16.5} \node at (\x,0.5) {$\overline{1}$};
					\node at (15.5,0.5) {$\cdots$};
					\draw[brace] (14.3,1.1) -- node[label,above] {$b_{4}$} (17,1.1);
					
					\draw (0,-1) rectangle (4,0);
					\foreach \x in {0.5,2.5} \node at (\x,-0.5) {2};
					\node at (1.5,-0.5) {$\cdots$};
					\draw[brace mirror] (0,-1.1) -- node[label,below,xshift=-3mm] {$m_{2}-\frac{1}{2}(b_{1}-1)-1$} (3,-1.1);
					
					\draw (3,-1) rectangle (8,0);
					\node at (3.5,-0.5) {0};
					\foreach \x in {4.5,6.5,7.5} \node at (\x,-0.5) {$\overline{2}$};
					\node at (5.5,-0.5) {$\cdots$};
					\draw[brace mirror] (4,-1.1) -- node[label,below] {$\frac{1}{2}(b_{1}-1)-\frac{1}{2}(b_{2}-m_{1}-1)$} (8,-1.1);
					
					\draw (8,-1) rectangle (13,0);
					\foreach \x in {9,12} \node at (\x,-0.5) {$\overline{1}$};
					\node at (10.5,-0.5) {$\cdots$};
					\draw[brace mirror] (8,-1.1) -- node[label,below] {$\frac{1}{2}(b_{2}-m_{1}-1)$} (13,-1.1);
					
				\end{tikzpicture}
				\caption{}
				\label{OSPKN12}
			\end{center}
		\end{figure}
		
		If $b_1$ is odd, $b_3$ is even, and $b_2 = b_1 + m_1$, the unique $T(\bm{b})$ such that $\psi(T(\bm{b})) = \bm{f^{b}}v_{\lambda}$ is shown in Figure~\ref{OSPKN13}.
		\begin{figure}[H]
			\begin{center}
				\begin{tikzpicture}[scale=0.7,
					every node/.style={font=\large},
					brace/.style={decorate,decoration={brace,  amplitude=5pt,raise=2pt}},
					brace mirror/.style={  decorate,decoration={brace,  mirror, amplitude=5pt,raise=2pt}},
					label/.style={midway,font=\scriptsize,outer sep=6pt} ]
					
					\draw (0,0) rectangle (4,1);
					\foreach \x in {0.5,2.5} \node at (\x,0.5) {1};
					\node at (1.5,0.5) {$\cdots$};
					
					\node at (3.5,0.5) {2};
					
					\draw[brace] (0,1.1) -- node[label,above,xshift=-4mm] {$m_{2}-\frac{1}{2}(b_{2}-m_{1}+1)$} (3,1.1);
					
					\draw (3,0) rectangle (14,1);
					\foreach \x in {4.9,6.9} \node at (\x,0.5) {2};
					\node at (5.9,0.5) {$\cdots$};
					\draw[brace] (3,1.1) -- node[label,above,xshift=6mm] {$\frac{1}{2}(b_{2}+m_{1}+1)-\frac{1}{2}b_3$} (7,1.1);
					
					\foreach \x in {8.3,10.3} \node at (\x,0.5) {$\overline{2}$};
					\node at (9.3,0.5) {$\cdots$};
					\draw[brace] (8,1.1) -- node[label,above] {$\frac{1}{2}b_{3}-b_{4}$} (10.6,1.1);
					
					\foreach \x in {11.5,13.5} \node at (\x,0.5) {$\overline{1}$};
					\node at (12.5,0.5) {$\cdots$};
					\draw[brace] (11.3,1.1) -- node[label,above] {$b_{4}$} (14,1.1);
					
					\draw (0,-1) rectangle (4,0);
					\foreach \x in {0.5,2.5} \node at (\x,-0.5) {2};
					\node at (1.5,-0.5) {$\cdots$};
					\draw[brace mirror] (0,-1.1) -- node[label,below] {$m_{2}-\frac{1}{2}(b_{1}-1)-1$} (3,-1.1);
					
					\node at (3.5,-0.5) {0};
					
					\draw (3,-1) rectangle (9,0);
					\foreach \x in {5,8} \node at (\x,-0.5) {$\overline{1}$};
					\node at (6.5,-0.5) {$\cdots$};
					\draw[brace mirror] (4,-1.1) -- node[label,below] {$\frac{1}{2}(b_{2}-m_{1}-1)$} (9,-1.1);
					
				\end{tikzpicture}
				\caption{}
				\label{OSPKN13}
			\end{center}
		\end{figure}
		
		\noindent \textbf{Case (iii)} $b_2 < m_1$:
		
		If $b_1$ and $b_3$ are even, the unique $T(\bm{b})$ such that $\psi(T(\bm{b})) = \bm{f^{b}}v_{\lambda}$ is shown in Figure~\ref{OSPKN14}. 
		\begin{figure}[H]
			\begin{center}
				\begin{tikzpicture}[scale=0.7,
					every node/.style={font=\large},
					brace/.style={decorate,decoration={brace,  amplitude=5pt,raise=2pt}},
					brace mirror/.style={  decorate,decoration={brace,  mirror, amplitude=5pt,raise=2pt}},
					label/.style={midway,font=\scriptsize,outer sep=6pt} ]
					
					\draw (0,0) rectangle (3,1);
					\foreach \x in {0.5,2.5} \node at (\x,0.5) {1};
					\node at (1.5,0.5) {$\cdots$};
					
					\draw (3,0) rectangle (6,1);
					\foreach \x in {3.5,5.5,6.5,8.5} \node at (\x,0.5) {1};
					\node at (4.5,0.5) {$\cdots$};
					\node at (7.5,0.5) {$\cdots$};
					
					\draw[brace] (0,1.1) -- node[label,above] {$m_{1}+m_{2}-b_2$} (8.5,1.1);
					
					\draw (6,0) rectangle (18,1);
					\foreach \x in {9.5,11.5} \node at (\x,0.5) {2};
					\node at (10.5,0.5) {$\cdots$};
					\draw[brace] (9.4,1.1) -- node[label,above] {$b_{2}-\frac{1}{2}b_3$} (11.6,1.1);
					
					\foreach \x in {12.5,14.5} \node at (\x,0.5) {$\overline{2}$};
					\node at (13.5,0.5) {$\cdots$};
					\draw[brace] (12.3,1.1) -- node[label,above] {$\frac{1}{2}b_{3}-b_{4}$} (14.6,1.1);
					
					\foreach \x in {15.5,17.5} \node at (\x,0.5) {$\overline{1}$};
					\node at (16.5,0.5) {$\cdots$};
					\draw[brace] (15.3,1.1) -- node[label,above] {$b_{4}$} (18,1.1);

					\draw (0,-1) rectangle (3,0);
					\foreach \x in {0.5,2.5} \node at (\x,-0.5) {2};
					\node at (1.5,-0.5) {$\cdots$};
					\draw[brace mirror] (0,-1.1) -- node[label,below] {$m_{2}-\frac{1}{2}b_{1}$} (3,-1.1);
					
					\draw (3,-1) rectangle (6,0);
					
					\foreach \x in {3.5,5.5} \node at (\x,-0.5) {$\overline{2}$};
					\node at (4.5,-0.5) {$\cdots$};
					\draw[brace mirror] (3,-1.1) -- node[label,below] {$\frac{1}{2}b_{1}$} (6,-1.1);
					
				\end{tikzpicture}
				\caption{}
				\label{OSPKN14}
			\end{center}
		\end{figure}
		
		If $b_1$ is even and $b_3$ is odd, the unique $T(\bm{b})$ such that $\psi(T(\bm{b})) = \bm{f^{b}}v_{\lambda}$ is shown in Figure~\ref{OSPKN15}. 
		\begin{figure}[H]
			\begin{center}
				\begin{tikzpicture}[scale=0.7,
					every node/.style={font=\large},
					brace/.style={decorate,decoration={brace,  amplitude=5pt,raise=2pt}},
					brace mirror/.style={  decorate,decoration={brace,  mirror, amplitude=5pt,raise=2pt}},
					label/.style={midway,font=\scriptsize,outer sep=6pt} ]
					
					\draw (0,0) rectangle (3,1);
					\foreach \x in {0.5,2.5} \node at (\x,0.5) {1};
					\node at (1.5,0.5) {$\cdots$};
					
					\draw (3,0) rectangle (6,1);
					\foreach \x in {3.5,5.5,6.5,8.5} \node at (\x,0.5) {1};
					\node at (4.5,0.5) {$\cdots$};
					\node at (7.5,0.5) {$\cdots$};
					
					\draw[brace] (0,1.1) -- node[label,above] {$m_{1}+m_{2}-b_2$} (8.5,1.1);
					
					\draw (6,0) rectangle (19,1);
					\foreach \x in {9.5,11.5} \node at (\x,0.5) {2};
					\node at (10.5,0.5) {$\cdots$};
					\draw[brace] (9.4,1.1) -- node[label,above] {$b_{2}-\frac{1}{2}(b_{3}-1)-1$} (11.6,1.1);
					
					\foreach \x in {13.5,15.5} \node at (\x,0.5) {$\overline{2}$};
					\node at (12.5,0.5) {0};
					\node at (14.5,0.5) {$\cdots$};
					\draw[brace] (13.3,1.1) -- node[label,above] {$\frac{1}{2}(b_{3}-1)-b_{4}$} (15.6,1.1);
					
					\foreach \x in {16.5,18.5} \node at (\x,0.5) {$\overline{1}$};
					\node at (17.5,0.5) {$\cdots$};
					\draw[brace] (16.3,1.1) -- node[label,above] {$b_{4}$} (19,1.1);

					\draw (0,-1) rectangle (3,0);
					\foreach \x in {0.5,2.5} \node at (\x,-0.5) {2};
					\node at (1.5,-0.5) {$\cdots$};
					\draw[brace mirror] (0,-1.1) -- node[label,below] {$m_{2}-\frac{1}{2}b_{1}$} (3,-1.1);
					
					\draw (3,-1) rectangle (6,0);
					
					\foreach \x in {3.5,5.5} \node at (\x,-0.5) {$\overline{2}$};
					\node at (4.5,-0.5) {$\cdots$};
					\draw[brace mirror] (3,-1.1) -- node[label,below] {$\frac{1}{2}b_{1}$} (6,-1.1);
					
				\end{tikzpicture}
				\caption{}
				\label{OSPKN15}
			\end{center}
		\end{figure}
		
		If $b_1$ and $b_3$ are odd, the unique $T(\bm{b})$ such that $\psi(T(\bm{b})) = \bm{f^{b}}v_{\lambda}$ is shown in Figure~\ref{OSPKN16}. 
		\begin{figure}[H]
			\begin{center}
				\begin{tikzpicture}[scale=0.7,
					every node/.style={font=\large},
					brace/.style={decorate,decoration={brace,  amplitude=5pt,raise=2pt}},
					brace mirror/.style={  decorate,decoration={brace,  mirror, amplitude=5pt,raise=2pt}},
					label/.style={midway,font=\scriptsize,outer sep=6pt} ]
					
					\draw (-1,0) rectangle (3,1);
					\foreach \x in {-0.5,1.5,2.5} \node at (\x,0.5) {1};
					\node at (0.5,0.5) {$\cdots$};
					
					\draw (2,0) rectangle (6,1);
					\foreach \x in {3.5,5.5,6.5,8.3} \node at (\x,0.5) {1};
					\node at (4.5,0.5) {$\cdots$};
					\node at (7.4,0.5) {$\cdots$};
					
					\draw[brace] (-1,1.1) -- node[label,above] {$m_{1}+m_{2}-b_2$} (8.3,1.1);
					
					\draw (6,0) rectangle (17.8,1);
					\foreach \x in {9.2,11} \node at (\x,0.5) {2};
					\node at (10.1,0.5) {$\cdots$};
					\draw[brace] (9,1.1) -- node[label,above,xshift=-1mm] {$b_{2}-\frac{1}{2}(b_{3}-1)-1$} (11.1,1.1);
					
					\foreach \x in {12.8,14.6} \node at (\x,0.5) {$\overline{2}$};
					\node at (11.9,0.5) {0};
					\node at (13.7,0.5) {$\cdots$};
					\draw[brace] (12.6,1.1) -- node[label,above] {$\frac{1}{2}(b_{3}-1)-b_{4}$} (14.7,1.1);
					
					\foreach \x in {15.5,17.3} \node at (\x,0.5) {$\overline{1}$};
					\node at (16.4,0.5) {$\cdots$};
					\draw[brace] (15.3,1.1) -- node[label,above] {$b_{4}$} (17.8,1.1);

					\draw (-1,-1) rectangle (3,0);
					\foreach \x in {-0.5,1.5} \node at (\x,-0.5) {2};
					\node at (0.5,-0.5) {$\cdots$};
					\node at (2.5,-0.5) {0};
					\draw[brace mirror] (-1,-1.1) -- node[label,below] {$m_{2}-\frac{1}{2}(b_{1}-1)-1$} (2,-1.1);
					
					\draw (2,-1) rectangle (6,0);
					
					\foreach \x in {3.5,5.5} \node at (\x,-0.5) {$\overline{2}$};
					\node at (4.5,-0.5) {$\cdots$};
					\draw[brace mirror] (3,-1.1) -- node[label,below] {$\frac{1}{2}(b_{1}-1)$} (6,-1.1);
					
				\end{tikzpicture}
				\caption{}
				\label{OSPKN16}
			\end{center}
		\end{figure}
		
		If $b_1$ is odd and $b_3$ is even, the unique $T(\bm{b})$ such that $\psi(T(\bm{b})) = \bm{f^{b}}v_{\lambda}$ is shown in Figure~\ref{OSPKN17}.
		\begin{figure}[H]
			\begin{center}
				\begin{tikzpicture}[scale=0.7,
					every node/.style={font=\large},
					brace/.style={decorate,decoration={brace,  amplitude=5pt,raise=2pt}},
					brace mirror/.style={  decorate,decoration={brace,  mirror, amplitude=5pt,raise=2pt}},
					label/.style={midway,font=\scriptsize,outer sep=6pt} ]
					
					\draw (-1,0) rectangle (3,1);
					\foreach \x in {-0.5,1.5,2.5} \node at (\x,0.5) {1};
					\node at (0.5,0.5) {$\cdots$};
					
					\draw (2,0) rectangle (6,1);
					\foreach \x in {3.5,5.5,6.5,8.5} \node at (\x,0.5) {1};
					\node at (4.5,0.5) {$\cdots$};
					\node at (7.5,0.5) {$\cdots$};
					
					\draw[brace] (-1,1.1) -- node[label,above] {$m_{1}+m_{2}-b_2$} (8.5,1.1);
					
					\draw (6,0) rectangle (18,1);
					\foreach \x in {9.5,11.5} \node at (\x,0.5) {2};
					\node at (10.5,0.5) {$\cdots$};
					\draw[brace] (9.4,1.1) -- node[label,above] {$b_{2}-\frac{1}{2}b_{3}$} (11.6,1.1);
					
					\foreach \x in {12.5,14.5} \node at (\x,0.5) {$\overline{2}$};
					\node at (13.5,0.5) {$\cdots$};
					\draw[brace] (12.4,1.1) -- node[label,above] {$\frac{1}{2}b_{3}-b_{4}$} (14.6,1.1);
					
					\foreach \x in {15.5,17.5} \node at (\x,0.5) {$\overline{1}$};
					\node at (16.5,0.5) {$\cdots$};
					\draw[brace] (15.4,1.1) -- node[label,above] {$b_{4}$} (18,1.1);

					\draw (-1,-1) rectangle (3,0);
					\foreach \x in {-0.5,1.5} \node at (\x,-0.5) {2};
					\node at (0.5,-0.5) {$\cdots$};
					\node at (2.5,-0.5) {0};
					\draw[brace mirror] (-1,-1.1) -- node[label,below] {$m_{2}-\frac{1}{2}(b_{1}-1)-1$} (2,-1.1);
					
					\draw (2,-1) rectangle (6,0);
					
					\foreach \x in {3.5,5.5} \node at (\x,-0.5) {$\overline{2}$};
					\node at (4.5,-0.5) {$\cdots$};
					\draw[brace mirror] (3,-1.1) -- node[label,below] {$\frac{1}{2}(b_{1}-1)$} (6,-1.1);
				\end{tikzpicture}
				\caption{}
				\label{OSPKN17}
			\end{center}
		\end{figure}
		
	Therefore, $\psi$ is a bijection, which completes the proof.
	\end{proof}
	
	Let \( T(\bm{b}) \) denote the KN tableau of $\mathfrak{spo}(4|1)$ corresponding to the Verma vector \( \bm{f^{b}}v_{\lambda} \). Then we have the following theorem.
	
	\begin{theorem}
		The weight of the Verma vector $\bm{f^{b}}v_{\lambda}$ is equal to the weight of its corresponding KN tableau $T(\bm{b})$ of $\mathfrak{spo}(4|1)$.
	\end{theorem}
	
	\begin{proof}
		Let \( \bm{b} = (b_{1}, b_{2}, b_{3}, b_{4}) \), and denote the weight of \( \bm{f^{b}}v_{\lambda} \) by \( \mu \). Then
		\begin{align*}
			\mu &=\lambda-b_1\epsilon_2-b_2(\epsilon_1-\epsilon_2)-b_3\epsilon_2-b_4(\epsilon_1-\epsilon_2)  
			\\&=(m_1+m_2)\epsilon_1+m_2\epsilon_2-b_1\epsilon_2-b_2(\epsilon_1-\epsilon_2)-b_3\epsilon_2-b_4(\epsilon_1-\epsilon_2)  
			\\&=(m_1+m_2-b_2-b_4)\epsilon_1+(m_2-b_1+b_2-b_3+b_4)\epsilon_2.
		\end{align*}
		Next, we calculate the weight of $T(\bm{b})$ according to the cases for $\bm{b}=(b_{1}, b_{2}, b_{3}, b_{4})$.
		
		\noindent \textbf{Case (i)} $b_2 \geq m_1$, and $\frac{1}{2}(b_2 - m_1) \in \mathbb{Z}_{\geq 0}$:
		
		If \( b_1 \) and \( b_3 \) are even, then \( T(\bm{b}) \) is as shown in Figure~\ref{OSPKN1}. It follows that
		\begin{align*}
			\mathrm{wt}(T(\bm{b}))&=\Big((m_2-\tfrac{1}{2}(b_2-m_1))-b_4-\tfrac{1}{2}(b_2-m_1)\Big)\epsilon_1 \\
			&\quad +\Big((\tfrac{1}{2}(b_2+m_1)-\tfrac{1}{2}b_3)+(m_2-\tfrac{1}{2}b_1) \\
			&\quad -(\tfrac{1}{2}b_3-b_4)-(\tfrac{1}{2}b_1-\tfrac{1}{2}(b_2-m_1))\Big)\epsilon_2 \\
			&=(m_1 + m_2 - b_2 - b_4)\epsilon_1 + (m_2 - b_1 + b_2 - b_3 + b_4)\epsilon_2.
		\end{align*}
		
		If \( b_1 \) is even and \( b_3 \) is odd, then \( T(\bm{b}) \) is as shown in Figure~\ref{OSPKN2}. It follows that
		\begin{align*}
			\mathrm{wt}(T(\bm{b}))&=\Big((m_2-\tfrac{1}{2}(b_2-m_1))-b_4-\tfrac{1}{2}(b_2-m_1)\Big)\epsilon_1 \\
			&\quad +\Big((\tfrac{1}{2}(b_2+m_1)-\tfrac{1}{2}(b_3-1)-1) +(m_2-\tfrac{1}{2}b_1) \\
			&\quad -(\tfrac{1}{2}(b_3-1)-b_4)-(\tfrac{1}{2}b_1-\tfrac{1}{2}(b_2-m_1))\Big)\epsilon_2 \\
			&=(m_1 + m_2 - b_2 - b_4)\epsilon_1 + (m_2 - b_1 + b_2 - b_3 + b_4)\epsilon_2.
		\end{align*}
		
		If \( b_1 \) and \( b_3 \) are odd, then \( T(\bm{b}) \) is as shown in Figure~\ref{OSPKN3}. It follows that
		\begin{align*}
			\mathrm{wt}(T(\bm{b}))&=\Big((m_2-\tfrac{1}{2}(b_2-m_1))-b_4-\tfrac{1}{2}(b_2-m_1)\Big)\epsilon_1 \\
			&\quad +\Big((\tfrac{1}{2}(b_2+m_1)-\tfrac{1}{2}(b_3-1)-1) +(m_2-\tfrac{1}{2}(b_1-1)-1) \\
			&\quad -(\tfrac{1}{2}(b_3-1)-b_4)-(\tfrac{1}{2}(b_1-1)-\tfrac{1}{2}(b_2-m_1))\Big)\epsilon_2 \\
			&=(m_1 + m_2 - b_2 - b_4)\epsilon_1 + (m_2 - b_1 + b_2 - b_3 + b_4)\epsilon_2.
		\end{align*}
		
		If \( b_1 \) is odd and \( b_3 \) is even, then \( T(\bm{b}) \) is as shown in Figure~\ref{OSPKN4}. It follows that
		\begin{align*}
			\mathrm{wt}(T(\bm{b}))&=\Big((m_2-\tfrac{1}{2}(b_2-m_1))-b_4-\tfrac{1}{2}(b_2-m_1)\Big)\epsilon_1 \\
			&\quad +\Big((\tfrac{1}{2}(b_2+m_1)-\tfrac{1}{2}b_3) +(m_2-\tfrac{1}{2}(b_1-1)-1) \\
			&\quad -(\tfrac{1}{2}b_3-b_4)-(\tfrac{1}{2}(b_1-1)-\tfrac{1}{2}(b_2-m_1))\Big)\epsilon_2 \\
			&=(m_1 + m_2 - b_2 - b_4)\epsilon_1 + (m_2 - b_1 + b_2 - b_3 + b_4)\epsilon_2.
		\end{align*}
		
		\noindent \textbf{Case (ii)} $b_2 \geq m_1$, and $\frac{1}{2}(b_2 - m_1 - 1) \in \mathbb{Z}_{\geq 0}$:
		
		If \( b_1 \) and \( b_3 \) are even, then \( T(\bm{b}) \) is as shown in Figure~\ref{OSPKN5}. It follows that
		\begin{align*}
			\mathrm{wt}(T(\bm{b}))&=\Big((m_2-\tfrac{1}{2}(b_2-m_1+1))-b_4-\tfrac{1}{2}(b_2-m_1-1)\Big)\epsilon_1 \\
			&\quad +\Big((\tfrac{1}{2}(b_2+m_1+1)-\tfrac{1}{2}b_3)+(m_2-\tfrac{1}{2}b_1) \\
			&\quad -(\tfrac{1}{2}b_3-b_4)-(\tfrac{1}{2}b_1-\tfrac{1}{2}(b_2-m_1-1))\Big)\epsilon_2 \\
			&=(m_1 + m_2 - b_2 - b_4)\epsilon_1 + (m_2 - b_1 + b_2 - b_3 + b_4)\epsilon_2.
		\end{align*}
		
		If \( b_1 \) is even, \( b_3 \) is odd, and \( b_3 \neq b_2 + m_1 \), then \( T(\bm{b}) \) is as shown in Figure~\ref{OSPKN6}. It follows that
		\begin{align*}
			\mathrm{wt}(T(\bm{b}))&=\Big((m_2-\tfrac{1}{2}(b_2-m_1+1))-b_4-\tfrac{1}{2}(b_2-m_1-1)\Big)\epsilon_1 \\
			&\quad +\Big((\tfrac{1}{2}(b_2+m_1+1)-\tfrac{1}{2}(b_3-1)-1) +(m_2-\tfrac{1}{2}b_1) \\
			&\quad -(\tfrac{1}{2}(b_3-1)-b_4)-(\tfrac{1}{2}b_1-\tfrac{1}{2}(b_2-m_1-1))\Big)\epsilon_2 \\
			&=(m_1 + m_2 - b_2 - b_4)\epsilon_1 + (m_2 - b_1 + b_2 - b_3 + b_4)\epsilon_2.
		\end{align*}
		
		If \( b_1 \) is even, \( b_3 \) is odd, and \( b_3 = b_2 + m_1 \), then \( T(\bm{b}) \) is as shown in Figure~\ref{OSPKN7}. It follows that
		\begin{align*}
			\mathrm{wt}(T(\bm{b}))&=\Big((m_2-\tfrac{1}{2}(b_2-m_1+1))-b_4-\tfrac{1}{2}(b_2-m_1-1)\Big)\epsilon_1 \\
			&\quad +\Big((m_2-\tfrac{1}{2}b_1)-(\tfrac{1}{2}(b_3-1)-b_4)  -(\tfrac{1}{2}b_1-\tfrac{1}{2}(b_2-m_1-1))\Big)\epsilon_2 \\
			&=(m_1 + m_2 - b_2 - b_4)\epsilon_1 + (m_2 - \tfrac{1}{2}m_1 + \tfrac{1}{2}b_2 - b_1 - \tfrac{1}{2}b_3 + b_4)\epsilon_2 \\
			&=(m_1 + m_2 - b_2 - b_4)\epsilon_1 + (m_2 - b_1 + b_2 - b_3 + b_4)\epsilon_2.
		\end{align*}
		
		If \( b_1 \) and \( b_3 \) are odd, with \( b_2 \neq b_1 + m_1 \) and \( b_3 \neq b_2 + m_1 \), then \( T(\bm{b}) \) is as shown in Figure~\ref{OSPKN8}. It follows that
		\begin{align*}
			\mathrm{wt}(T(\bm{b}))&=\Big((m_2-\tfrac{1}{2}(b_2-m_1+1))-b_4-\tfrac{1}{2}(b_2-m_1-1)\Big)\epsilon_1 \\
			&\quad +\Big((\tfrac{1}{2}(b_2+m_1+1)-\tfrac{1}{2}(b_3-1)-1) \\
			&\quad +(m_2-\tfrac{1}{2}(b_1-1)-1)-(\tfrac{1}{2}(b_3-1)-b_4) \\
			&\quad -(\tfrac{1}{2}(b_1-1)-\tfrac{1}{2}(b_2-m_1-1))\Big)\epsilon_2 \\
			&=(m_1 + m_2 - b_2 - b_4)\epsilon_1 + (m_2 - b_1 + b_2 - b_3 + b_4)\epsilon_2.
		\end{align*}
		
		If \( b_1 \) and \( b_3 \) are odd, with \( b_2 \neq b_1 + m_1 \) and \( b_3 = b_2 + m_1 \), then \( T(\bm{b}) \) is as shown in Figure~\ref{OSPKN9}. It follows that
		\begin{align*}
			\mathrm{wt}(T(\bm{b}))&=\Big((m_2-\tfrac{1}{2}(b_2-m_1+1))-b_4-\tfrac{1}{2}(b_2-m_1-1)\Big)\epsilon_1 \\
			&\quad +\Big((m_2-\tfrac{1}{2}(b_1-1)-1)-(\tfrac{1}{2}(b_3-1)-b_4) \\
			&\quad -(\tfrac{1}{2}(b_1-1)-\tfrac{1}{2}(b_2-m_1-1))\Big)\epsilon_2 \\
			&=(m_1 + m_2 - b_2 - b_4)\epsilon_1 + (m_2 - b_1 + b_2 - b_3 + b_4)\epsilon_2.
		\end{align*}
		
		If \( b_1 \) and \( b_3 \) are odd, with \( b_2 = b_1 + m_1 \) and \( b_3 \neq b_2 + m_1 \), then \( T(\bm{b}) \) is as shown in Figure~\ref{OSPKN10}. It follows that
		\begin{align*}
			\mathrm{wt}(T(\bm{b}))&=\Big((m_2-\tfrac{1}{2}(b_2-m_1+1))-b_4-\tfrac{1}{2}(b_2-m_1-1)\Big)\epsilon_1 \\
			&\quad +\Big((\tfrac{1}{2}(b_2+m_1+1)-\tfrac{1}{2}(b_3-1)-1) \\
			&\quad +(m_2-\tfrac{1}{2}(b_1-1)-1)-(\tfrac{1}{2}(b_3-1)-b_4)\Big)\epsilon_2 \\
			&=(m_1 + m_2 - b_2 - b_4)\epsilon_1 + (m_2 - b_1 + b_2 - b_3 + b_4)\epsilon_2.
		\end{align*}
		
		If \( b_1 \) and \( b_3 \) are odd, with \( b_2 = b_1 + m_1 \) and \( b_3 = b_2 + m_1 \), then \( T(\bm{b}) \) is as shown in Figure~\ref{OSPKN11}. It follows that
		\begin{align*}
			\mathrm{wt}(T(\bm{b}))&=\Big((m_2-\tfrac{1}{2}(b_2-m_1+1))-b_4-\tfrac{1}{2}(b_2-m_1-1)\Big)\epsilon_1 \\
			&\quad +\Big((m_2-\tfrac{1}{2}(b_1-1)-1)-(\tfrac{1}{2}(b_3-1)-b_4)\Big)\epsilon_2 \\
			&=(m_1 + m_2 - b_2 - b_4)\epsilon_1 + (m_2 - b_1 + b_2 - b_3 + b_4)\epsilon_2.
		\end{align*}
		
		If \( b_1 \) is odd, \( b_3 \) is even, and \( b_2 \neq b_1 + m_1 \), then \( T(\bm{b}) \) is as shown in Figure~\ref{OSPKN12}. It follows that
		\begin{align*}
			\mathrm{wt}(T(\bm{b}))&=\Big((m_2-\tfrac{1}{2}(b_2-m_1+1))-b_4-\tfrac{1}{2}(b_2-m_1-1)\Big)\epsilon_1 \\
			&\quad +\Big((\tfrac{1}{2}(b_2+m_1+1)-\tfrac{1}{2}b_3) +(m_2-\tfrac{1}{2}(b_1-1)-1) \\
			&\quad -(\tfrac{1}{2}b_3-b_4)-(\tfrac{1}{2}(b_1-1)-\tfrac{1}{2}(b_2-m_1-1))\Big)\epsilon_2 \\
			&=(m_1 + m_2 - b_2 - b_4)\epsilon_1 + (m_2 - b_1 + b_2 - b_3 + b_4)\epsilon_2.
		\end{align*}
		
		If \( b_1 \) is odd, \( b_3 \) is even, and \( b_2 = b_1 + m_1 \), then \( T(\bm{b}) \) is as shown in Figure~\ref{OSPKN13}. It follows that
		\begin{align*}
			\mathrm{wt}(T(\bm{b}))&=\Big((m_2-\tfrac{1}{2}(b_2-m_1+1))-b_4-\tfrac{1}{2}(b_2-m_1-1)\Big)\epsilon_1 \\
			&\quad +\Big((\tfrac{1}{2}(b_2+m_1+1)-\tfrac{1}{2}b_3) +(m_2-\tfrac{1}{2}(b_1-1)-1)-(\tfrac{1}{2}b_3-b_4)\Big)\epsilon_2 \\
			&=(m_1 + m_2 - b_2 - b_4)\epsilon_1 + (m_2 - b_1 + b_2 - b_3 + b_4)\epsilon_2.
		\end{align*}
		
		\noindent \textbf{Case (iii)} $b_2 < m_1$:
		
		If \( b_1, b_3 \) are even, then \( T(\bm{b}) \) is as shown in Figure~\ref{OSPKN14}. It follows that
		\begin{align*}
			\mathrm{wt}(T(\bm{b}))&=\Big((m_1+m_2-b_2)-b_4\Big)\epsilon_1 \\
			&\quad +\Big((b_2-\tfrac{1}{2}b_3)+(m_2-\tfrac{1}{2}b_1)-(\tfrac{1}{2}b_3-b_4)-\tfrac{1}{2}b_1\Big)\epsilon_2 \\
			&=(m_1 + m_2 - b_2 - b_4)\epsilon_1 + (m_2 - b_1 + b_2 - b_3 + b_4)\epsilon_2.
		\end{align*}
		
		If \( b_1 \) is even and \( b_3 \) is odd, then \( T(\bm{b}) \) is as shown in Figure~\ref{OSPKN15}. It follows that
		\begin{align*}
			\mathrm{wt}(T(\bm{b}))&=\Big((m_1+m_2-b_2)-b_4\Big)\epsilon_1 \\
			&\quad +\Big((b_2-\tfrac{1}{2}(b_3-1)-1) +(m_2-\tfrac{1}{2}b_1)-(\tfrac{1}{2}(b_3-1)-b_4)-\tfrac{1}{2}b_1\Big)\epsilon_2 \\
			&=(m_1 + m_2 - b_2 - b_4)\epsilon_1 + (m_2 - b_1 + b_2 - b_3 + b_4)\epsilon_2.
		\end{align*}
		
		If \(b_1, b_3\) are odd, then \(T(\bm{b})\) is as shown in Figure~\ref{OSPKN16}. It follows that
		\begin{align*}
			\mathrm{wt}(T(\bm{b}))&=\Big((m_1+m_2-b_2)-b_4\Big)\epsilon_1 \\
			&\quad +\Big((b_2-\tfrac{1}{2}(b_3-1)-1) +(m_2-\tfrac{1}{2}(b_1-1)-1) \\
			&\quad -(\tfrac{1}{2}(b_3-1)-b_4)-\tfrac{1}{2}(b_1-1)\Big)\epsilon_2 \\
			&=(m_1 + m_2 - b_2 - b_4)\epsilon_1 + (m_2 - b_1 + b_2 - b_3 + b_4)\epsilon_2.
		\end{align*}
		
		If \(b_1\) is odd and \(b_3\) is even, then \(T(\bm{b})\) is as shown in Figure~\ref{OSPKN17}. It follows that
		\begin{align*}
			\mathrm{wt}(T(\bm{b}))&=\Big((m_1+m_2-b_2)-b_4\Big)\epsilon_1 \\
			&\quad +\Big((b_2-\tfrac{1}{2}b_3)+(m_2-\tfrac{1}{2}(b_1-1)-1) -(\tfrac{1}{2}b_3-b_4)-\tfrac{1}{2}(b_1-1)\Big)\epsilon_2 \\
			&=(m_1 + m_2 - b_2 - b_4)\epsilon_1 + (m_2 - b_1 + b_2 - b_3 + b_4)\epsilon_2.
		\end{align*}
		
		In summary, we have $\mathrm{wt}(T(\bm{b})) = \mu$.
	\end{proof}

	\begin{example}\label{example}
		Let the highest weight of the irreducible representation $L(\lambda)$ of $\mathfrak{spo}(4|1)$ be $\lambda = \omega_{1} + 4\omega_{2} = 3\epsilon_{1} + 2\epsilon_{2}$. The Verma vectors of $L(\lambda)$ have the form
		$
		f_1^{b_4} f_2^{b_3} f_1^{b_2} f_2^{b_1} v_{\lambda},
		$
		where the parameters satisfy
		\[ 
		0 \leq b_{1} \leq 4, \quad 
		0 \leq b_{2} \leq 1 + b_{1}, \quad 
		0 \leq b_{3} \leq \min\{b_{2} + 1, 2b_{2}\}, \quad 
		0 \leq b_{4} \leq \min\{1, \tfrac{1}{2}b_{3}\}. 
		\]
		
		These Verma vectors can be represented by KN tableaux of $\mathfrak{spo}(4|1)$ as follows:
		
\begin{center}
	\renewcommand{\arraystretch}{2.5}
	\setlength{\tabcolsep}{12pt} 
	\begin{longtable}{ccccc}

		\YoungWithData{
			\begin{ytableau}
				1 & 1 & 1 \\
				2 & 2
			\end{ytableau}
		}{v_{\lambda}}
		&
		\YoungWithData{
			\begin{ytableau}
				1 & 1 & 1 \\
				2 & 0
			\end{ytableau}
		}{f_2 v_{\lambda}}
		&
		\YoungWithData{
			\begin{ytableau}
				1 & 1 & 1 \\
				2 & \overline{2}
			\end{ytableau}
		}{f_2^2 v_{\lambda}}
		&
		\YoungWithData{
			\begin{ytableau}
				1 & 1 & 1 \\
				0 & \overline{2}
			\end{ytableau}
		}{f_2^3 v_{\lambda}}
		&
		\YoungWithData{
			\begin{ytableau}
				1 & 1 & 1 \\
				\overline{2} & \overline{2}
			\end{ytableau}
		}{f_2^4 v_{\lambda}} \\[12pt]

		\YoungWithData{
			\begin{ytableau}
				1 & 1 & 2 \\
				2 & 2
			\end{ytableau}
		}{f_1 v_{\lambda}}
		&
		\YoungWithData{
			\begin{ytableau}
				1 & 1 & 2 \\
				2 & 0
			\end{ytableau}
		}{f_1 f_2 v_{\lambda}}
		&
		\YoungWithData{
			\begin{ytableau}
				1 & 1 & 2 \\
				2 & \overline{2}
			\end{ytableau}
		}{f_1 f_2^2 v_{\lambda}}
		&
		\YoungWithData{
			\begin{ytableau}
				1 & 1 & 2 \\
				0 & \overline{2}
			\end{ytableau}
		}{f_1 f_2^3 v_{\lambda}}
		&
		\YoungWithData{
			\begin{ytableau}
				1 & 1 & 2 \\
				\overline{2} & \overline{2}
			\end{ytableau}
		}{f_1 f_2^4 v_{\lambda}} \\[12pt]

		\YoungWithData{
			\begin{ytableau}
				1 & 1 & 0 \\
				2 & 2
			\end{ytableau}
		}{f_2 f_1 v_{\lambda}}
		&
		\YoungWithData{
			\begin{ytableau}
				1 & 1 & 0 \\
				2 & 0
			\end{ytableau}
		}{f_2 f_1 f_2 v_{\lambda}}
		&
		\YoungWithData{
			\begin{ytableau}
				1 & 1 & 0 \\
				2 & \overline{2}
			\end{ytableau}
		}{f_2 f_1 f_2^2 v_{\lambda}}
		&
		\YoungWithData{
			\begin{ytableau}
				1 & 1 & 0 \\
				0 & \overline{2}
			\end{ytableau}
		}{f_2 f_1 f_2^3 v_{\lambda}}
		&
		\YoungWithData{
			\begin{ytableau}
				1 & 1 & 0 \\
				\overline{2} & \overline{2}
			\end{ytableau}
		}{f_2 f_1 f_2^4 v_{\lambda}} \\[12pt]

		\YoungWithData{
			\begin{ytableau}
				1 & 1 & \overline{2} \\
				2 & 2
			\end{ytableau}
		}{f_2^2 f_1 v_{\lambda}}
		&
		\YoungWithData{
			\begin{ytableau}
				1 & 1 & \overline{2} \\
				2 & 0
			\end{ytableau}
		}{f_2^2 f_1 f_2 v_{\lambda}}
		&
		\YoungWithData{
			\begin{ytableau}
				1 & 1 & \overline{2} \\
				2 & \overline{2}
			\end{ytableau}
		}{f_2^2 f_1 f_2^2 v_{\lambda}}
		&
		\YoungWithData{
			\begin{ytableau}
				1 & 1 & \overline{2} \\
				0 & \overline{2}
			\end{ytableau}
		}{f_2^2 f_1 f_2^3 v_{\lambda}}
		&
		\YoungWithData{
			\begin{ytableau}
				1 & 1 & \overline{2} \\
				\overline{2} & \overline{2}
			\end{ytableau}
		}{f_2^2 f_1 f_2^4 v_{\lambda}} \\[12pt]

		\YoungWithData{
			\begin{ytableau}
				1 & 1 & \overline{1} \\
				2 & 2
			\end{ytableau}
		}{f_1 f_2^2 f_1 v_{\lambda}}
		&
		\YoungWithData{
			\begin{ytableau}
				1 & 1 & \overline{1} \\
				2 & 0
			\end{ytableau}
		}{f_1 f_2^2 f_1 f_2 v_{\lambda}}
		&
		\YoungWithData{
			\begin{ytableau}
				1 & 1 & \overline{1} \\
				2 & \overline{2}
			\end{ytableau}
		}{f_1 f_2^2 f_1 f_2^2 v_{\lambda}}
		&
		\YoungWithData{
			\begin{ytableau}
				1 & 1 & \overline{1} \\
				0 & \overline{2}
			\end{ytableau}
		}{f_1 f_2^2 f_1 f_2^3 v_{\lambda}}
		&
		\YoungWithData{
			\begin{ytableau}
				1 & 1 & \overline{1} \\
				\overline{2} & \overline{2}
			\end{ytableau}
		}{f_{1}f_{2}^{2}f_{1}f_{2}^{4}v_{\lambda}} \\[12pt]

		\YoungWithData{
			\begin{ytableau}
				1 & 2 & 2 \\
				2 & 0
			\end{ytableau}
		}{f_{1}^{2}f_{2}v_{\lambda}}
		&
		\YoungWithData{
			\begin{ytableau}
				1 & 2 & 2 \\
				2 & \overline{2}
			\end{ytableau}
		}{f_{1}^{2}f_{2}^{2}v_{\lambda}}
		&
		\YoungWithData{
			\begin{ytableau}
				1 & 2 & 2 \\
				2 & \overline{1}
			\end{ytableau}
		}{f_{1}^{3}f_{2}^{2}v_{\lambda}}
		&
		\YoungWithData{
			\begin{ytableau}
				1 & 2 & 2 \\
				0 & \overline{2}
			\end{ytableau}
		}{f_{1}^{2}f_{2}^{3}v_{\lambda}}
		&
		\YoungWithData{
			\begin{ytableau}
				1 & 2 & 2 \\
				0 & \overline{1}
			\end{ytableau}
		}{f_{1}^{3}f_{2}^{3}v_{\lambda}} \\[12pt]

		\YoungWithData{
			\begin{ytableau}
				1 & 2 & 2 \\
				\overline{2} & \overline{2}
			\end{ytableau}
		}{f_{1}^{2}f_{2}^{4}v_{\lambda}}
		&
		\YoungWithData{
			\begin{ytableau}
				1 & 2 & 2 \\
				\overline{2} & \overline{1}
			\end{ytableau}
		}{f_{1}^{3}f_{2}^{4}v_{\lambda}}
		&
		\YoungWithData{
			\begin{ytableau}
				1 & 2 & 0 \\
				2 & 0
			\end{ytableau}
		}{f_{2}f_{1}^{2}f_{2}v_{\lambda}}
		&
		\YoungWithData{
			\begin{ytableau}
				1 & 2 & 0 \\
				2 & \overline{2}
			\end{ytableau}
		}{f_{2}f_{1}^{2}f_{2}^{2}v_{\lambda}}
		&
		\YoungWithData{
			\begin{ytableau}
				1 & 2 & 0 \\
				2 & \overline{1}
			\end{ytableau}
		}{f_{2}f_{1}^{3}f_{2}^{2}v_{\lambda}} \\[12pt]

		\YoungWithData{
			\begin{ytableau}
				1 & 2 & 0 \\
				0 & \overline{2}
			\end{ytableau}
		}{f_{2}f_{1}^{2}f_{2}^{3}v_{\lambda}}
		&
		\YoungWithData{
			\begin{ytableau}
				1 & 2 & 0 \\
				0 & \overline{1}
			\end{ytableau}
		}{f_{2}f_{1}^{3}f_{2}^{3}v_{\lambda}}
		&
		\YoungWithData{
			\begin{ytableau}
				1 & 2 & 0 \\
				\overline{2} & \overline{2}
			\end{ytableau}
		}{f_{2}f_{1}^{2}f_{2}^{4}v_{\lambda}}
		&
		\YoungWithData{
			\begin{ytableau}
				1 & 2 & 0 \\
				\overline{2} & \overline{1}
			\end{ytableau}
		}{f_{2}f_{1}^{3}f_{2}^{4}v_{\lambda}}
		&
		\YoungWithData{
			\begin{ytableau}
				1 & 2 & \overline{2} \\
				2 & 0
			\end{ytableau}
		}{f_{2}^{2}f_{1}^{2}f_{2}v_{\lambda}} \\[12pt]

		\YoungWithData{
			\begin{ytableau}
				1 & 2 & \overline{2} \\
				2 & \overline{2}
			\end{ytableau}
		}{f_{2}^{2} f_{1}^{2} f_{2}^{2} v_{\lambda}}
		&
		\YoungWithData{
			\begin{ytableau}
				1 & 2 & \overline{2} \\
				2 & \overline{1}
			\end{ytableau}
		}{f_{2}^{2} f_{1}^{3} f_{2}^{2} v_{\lambda}}
		&
		\YoungWithData{
			\begin{ytableau}
				1 & 2 & \overline{2} \\
				0 & \overline{2}
			\end{ytableau}
		}{f_{2}^{2} f_{1}^{2} f_{2}^{3} v_{\lambda}}
		&
		\YoungWithData{
			\begin{ytableau}
				1 & 2 & \overline{2} \\
				0 & \overline{1}
			\end{ytableau}
		}{f_{2}^{2} f_{1}^{3} f_{2}^{3} v_{\lambda}}
		&
		\YoungWithData{
			\begin{ytableau}
				1 & 2 & \overline{2} \\
				\overline{2} & \overline{2}
			\end{ytableau}
		}{f_{2}^{2} f_{1}^{2} f_{2}^{4} v_{\lambda}} \\[12pt]

		\YoungWithData{
			\begin{ytableau}
				1 & 2 & \overline{2} \\
				\overline{2} & \overline{1}
			\end{ytableau}
		}{f_{2}^{2} f_{1}^{3} f_{2}^{4} v_{\lambda}}
		&
		\YoungWithData{
			\begin{ytableau}
				1 & 2 & \overline{1} \\
				2 & 0
			\end{ytableau}
		}{f_{1} f_{2}^{2} f_{1}^{2} f_{2} v_{\lambda}}
		&
		\YoungWithData{
			\begin{ytableau}
				1 & 2 & \overline{1} \\
				2 & \overline{2}
			\end{ytableau}
		}{f_{1} f_{2}^{2} f_{1}^{2} f_{2}^{2} v_{\lambda}}
		&
		\YoungWithData{
			\begin{ytableau}
				1 & 2 & \overline{1} \\
				2 & \overline{1}
			\end{ytableau}
		}{f_{1} f_{2}^{2} f_{1}^{3} f_{2}^{2} v_{\lambda}}
		&
		\YoungWithData{
			\begin{ytableau}
				1 & 2 & \overline{1} \\
				0 & \overline{2}
			\end{ytableau}
		}{f_{1} f_{2}^{2} f_{1}^{2} f_{2}^{3} v_{\lambda}} \\[12pt]

		\YoungWithData{
			\begin{ytableau}
				1 & 2 & \overline{1} \\
				0 & \overline{1}
			\end{ytableau}
		}{f_{1} f_{2}^{2} f_{1}^{3} f_{2}^{3} v_{\lambda}}
		&
		\YoungWithData{
			\begin{ytableau}
				1 & 2 & \overline{1} \\
				\overline{2} & \overline{2}
			\end{ytableau}
		}{f_{1} f_{2}^{2} f_{1}^{2} f_{2}^{4} v_{\lambda}}
		&
		\YoungWithData{
			\begin{ytableau}
				1 & 2 & \overline{1} \\
				\overline{2} & \overline{1}
			\end{ytableau}
		}{f_{1} f_{2}^{2} f_{1}^{3} f_{2}^{4} v_{\lambda}}
		&
		\YoungWithData{
			\begin{ytableau}
				1 & 0 & \overline{2} \\
				2 & 0
			\end{ytableau}
		}{f_{2}^{3} f_{1}^{2} f_{2} v_{\lambda}}
		&
		\YoungWithData{
			\begin{ytableau}
				1 & 0 & \overline{2} \\
				2 & \overline{2}
			\end{ytableau}
		}{f_{2}^{3} f_{1}^{2} f_{2}^{2} v_{\lambda}} \\[12pt]

		\YoungWithData{
			\begin{ytableau}
				1 & 0 & \overline{2} \\
				2 & \overline{1}
			\end{ytableau}
		}{f_{2}^{3} f_{1}^{3} f_{2}^{2} v_{\lambda}}
		&
		\YoungWithData{
			\begin{ytableau}
				1 & 0 & \overline{2} \\
				0 & \overline{2}
			\end{ytableau}
		}{f_{2}^{3} f_{1}^{2} f_{2}^{3} v_{\lambda}}
		&
		\YoungWithData{
			\begin{ytableau}
				1 & 0 & \overline{2} \\
				0 & \overline{1}
			\end{ytableau}
		}{f_{2}^{3} f_{1}^{3} f_{2}^{3} v_{\lambda}}
		&
		\YoungWithData{
			\begin{ytableau}
				1 & 0 & \overline{2} \\
				\overline{2} & \overline{2}
			\end{ytableau}
		}{f_{2}^{3} f_{1}^{2} f_{2}^{4} v_{\lambda}}
		&
		\YoungWithData{
			\begin{ytableau}
				1 & 0 & \overline{2} \\
				\overline{2} & \overline{1}
			\end{ytableau}
		}{f_{2}^{3} f_{1}^{3} f_{2}^{4} v_{\lambda}} \\[12pt]

		\YoungWithData{
			\begin{ytableau}
				1 & 0 & \overline{1} \\
				2 & 0
			\end{ytableau}
		}{f_{1} f_{2}^{3} f_{1}^{2} f_{2} v_{\lambda}}
		&
		\YoungWithData{
			\begin{ytableau}
				1 & 0 & \overline{1} \\
				2 & \overline{2}
			\end{ytableau}
		}{f_{1} f_{2}^{3} f_{1}^{2} f_{2}^{2} v_{\lambda}}
		&
		\YoungWithData{
			\begin{ytableau}
				1 & 0 & \overline{1} \\
				2 & \overline{1}
			\end{ytableau}
		}{f_{1} f_{2}^{3} f_{1}^{3} f_{2}^{2} v_{\lambda}}
		&
		\YoungWithData{
			\begin{ytableau}
				1 & 0 & \overline{1} \\
				0 & \overline{2}
			\end{ytableau}
		}{f_{1} f_{2}^{3} f_{1}^{2} f_{2}^{3} v_{\lambda}}
		&
		\YoungWithData{
			\begin{ytableau}
				1 & 0 & \overline{1} \\
				0 & \overline{1}
			\end{ytableau}
		}{f_{1} f_{2}^{3} f_{1}^{3} f_{2}^{3} v_{\lambda}} \\[12pt]

		\YoungWithData{
			\begin{ytableau}
				1 & 0 & \overline{1} \\
				\overline{2} & \overline{2}
			\end{ytableau}
		}{f_{1} f_{2}^{3} f_{1}^{2} f_{2}^{4} v_{\lambda}}
		&
		\YoungWithData{
			\begin{ytableau}
				1 & 0 & \overline{1} \\
				\overline{2} & \overline{1}
			\end{ytableau}
		}{f_{1} f_{2}^{3} f_{1}^{3} f_{2}^{4} v_{\lambda}}
		&
		\YoungWithData{
			\begin{ytableau}
				1 & \overline{2} & \overline{2} \\
				2 & \overline{1}
			\end{ytableau}
		}{f_{2}^{4} f_{1}^{3} f_{2}^{2} v_{\lambda}}
		&
		\YoungWithData{
			\begin{ytableau}
				1 & \overline{2} & \overline{2} \\
				0 & \overline{1}
			\end{ytableau}
		}{f_{2}^{4} f_{1}^{3} f_{2}^{3} v_{\lambda}}
		&
		\YoungWithData{
			\begin{ytableau}
				1 & \overline{2} & \overline{2} \\
				\overline{2} & \overline{1}
			\end{ytableau}
		}{f_{2}^{4} f_{1}^{3} f_{2}^{4} v_{\lambda}} \\[12pt]

		\YoungWithData{
			\begin{ytableau}
				1 & \overline{2} & \overline{1} \\
				2 & \overline{1}
			\end{ytableau}
		}{f_{1} f_{2}^{4} f_{1}^{3} f_{2}^{2} v_{\lambda}}
		&
		\YoungWithData{
			\begin{ytableau}
				1 & \overline{2} & \overline{1} \\
				0 & \overline{1}
			\end{ytableau}
		}{f_{1} f_{2}^{4} f_{1}^{3} f_{2}^{3} v_{\lambda}}
		&
		\YoungWithData{
			\begin{ytableau}
				1 & \overline{2} & \overline{1} \\
				\overline{2} & \overline{1}
			\end{ytableau}
		}{f_{1} f_{2}^{4} f_{1}^{3} f_{2}^{4} v_{\lambda}}
		&
		\YoungWithData{
			\begin{ytableau}
				2 & 2 & 2 \\
				0 & \overline{1}
			\end{ytableau}
		}{f_{1}^{4} f_{2}^{3} v_{\lambda}}
		&
		\YoungWithData{
			\begin{ytableau}
				2 & 2 & 2 \\
				\overline{2} & \overline{1}
			\end{ytableau}
		}{f_{1}^{4} f_{2}^{4} v_{\lambda}} \\[12pt]

		\YoungWithData{
			\begin{ytableau}
				2 & 2 & 2 \\
				\overline{1} & \overline{1}
			\end{ytableau}
		}{f_{1}^{5} f_{2}^{4} v_{\lambda}}
		&
		\YoungWithData{
			\begin{ytableau}
				2 & 2 & 0 \\
				0 & \overline{1}
			\end{ytableau}
		}{f_{2}f_{1}^{4}f_{2}^{3}v_{\lambda}}
		&
		\YoungWithData{
			\begin{ytableau}
				2 & 2 & 0 \\
				\overline{2} & \overline{1}
			\end{ytableau}
		}{f_{2}f_{1}^{4}f_{2}^{4}v_{\lambda}}
		&
		\YoungWithData{
			\begin{ytableau}
				2 & 2 & 0 \\
				\overline{1} & \overline{1}
			\end{ytableau}
		}{f_{2}f_{1}^{5}f_{2}^{4}v_{\lambda}}
		&
		\YoungWithData{
			\begin{ytableau}
				2 & 2 & \overline{2} \\
				0 & \overline{1}
			\end{ytableau}
		}{f_{2}^{2}f_{1}^{4}f_{2}^{3}v_{\lambda}} \\[12pt]

		\YoungWithData{
			\begin{ytableau}
				2 & 2 & \overline{2} \\
				\overline{2} & \overline{1}
			\end{ytableau}
		}{f_{2}^{2}f_{1}^{4}f_{2}^{4}v_{\lambda}}
		&
		\YoungWithData{
			\begin{ytableau}
				2 & 2 & \overline{2} \\
				\overline{1} & \overline{1}
			\end{ytableau}
		}{f_{2}^{2}f_{1}^{5}f_{2}^{4}v_{\lambda}}
		&
		\YoungWithData{
			\begin{ytableau}
				2 & 2 & \overline{1} \\
				0 & \overline{1}
			\end{ytableau}
		}{f_{1}f_{2}^{2}f_{1}^{4}f_{2}^{3}v_{\lambda}}
		&
		\YoungWithData{
			\begin{ytableau}
				2 & 2 & \overline{1} \\
				\overline{2} & \overline{1}
			\end{ytableau}
		}{f_{1}f_{2}^{2}f_{1}^{4}f_{2}^{4}v_{\lambda}}
		&
		\YoungWithData{
			\begin{ytableau}
				2 & 2 & \overline{1} \\
				\overline{1} & \overline{1}
			\end{ytableau}
		}{f_{1}f_{2}^{2}f_{1}^{5}f_{2}^{4}v_{\lambda}} \\[12pt]

		\YoungWithData{
			\begin{ytableau}
				2 & 0 & \overline{2} \\
				0 & \overline{1}
			\end{ytableau}
		}{f_{2}^{3}f_{1}^{4}f_{2}^{3}v_{\lambda}}
		&
		\YoungWithData{
			\begin{ytableau}
				2 & 0 & \overline{2} \\
				\overline{2} & \overline{1}
			\end{ytableau}
		}{f_{2}^{3}f_{1}^{4}f_{2}^{4}v_{\lambda}}
		&
		\YoungWithData{
			\begin{ytableau}
				2 & 0 & \overline{2} \\
				\overline{1} & \overline{1}
			\end{ytableau}
		}{f_{2}^{3}f_{1}^{5}f_{2}^{4}v_{\lambda}}
		&
		\YoungWithData{
			\begin{ytableau}
				2 & 0 & \overline{1} \\
				0 & \overline{1}
			\end{ytableau}
		}{f_{1}f_{2}^{3}f_{1}^{4}f_{2}^{3}v_{\lambda}}
		&
		\YoungWithData{
			\begin{ytableau}
				2 & 0 & \overline{1} \\
				\overline{2} & \overline{1}
			\end{ytableau}
		}{f_{1}f_{2}^{3}f_{1}^{4}f_{2}^{4}v_{\lambda}} \\[12pt]

		\YoungWithData{
			\begin{ytableau}
				2 & 0 & \overline{1} \\
				\overline{1} & \overline{1}
			\end{ytableau}
		}{f_{1}f_{2}^{3}f_{1}^{5}f_{2}^{4}v_{\lambda}}
		&
		\YoungWithData{
			\begin{ytableau}
				2 & \overline{2} & \overline{2} \\
				0 & \overline{1}
			\end{ytableau}
		}{f_{2}^{4}f_{1}^{4}f_{2}^{3}v_{\lambda}}
		&
		\YoungWithData{
			\begin{ytableau}
				2 & \overline{2} & \overline{2} \\
				\overline{2} & \overline{1}
			\end{ytableau}
		}{f_{2}^{4}f_{1}^{4}f_{2}^{4}v_{\lambda}}
		&
		\YoungWithData{
			\begin{ytableau}
				2 & \overline{2} & \overline{2} \\
				\overline{1} & \overline{1}
			\end{ytableau}
		}{f_{2}^{4}f_{1}^{5}f_{2}^{4}v_{\lambda}}
		&
		\YoungWithData{
			\begin{ytableau}
				2 & \overline{2} & \overline{1} \\
				0 & \overline{1}
			\end{ytableau}
		}{f_{1}f_{2}^{4}f_{1}^{4}f_{2}^{3}v_{\lambda}} \\[12pt]

		\YoungWithData{
			\begin{ytableau}
				2 & \overline{2} & \overline{1} \\
				\overline{2} & \overline{1}
			\end{ytableau}
		}{f_{1}f_{2}^{4}f_{1}^{4}f_{2}^{4}v_{\lambda}}
		&
		\YoungWithData{
			\begin{ytableau}
				2 & \overline{2} & \overline{1} \\
				\overline{1} & \overline{1}
			\end{ytableau}
		}{f_{1}f_{2}^{4}f_{1}^{5}f_{2}^{4}v_{\lambda}}
		&
		\YoungWithData{
			\begin{ytableau}
				0 & \overline{2} & \overline{2} \\
				0 & \overline{1}
			\end{ytableau}
		}{f_{2}^{5}f_{1}^{4}f_{2}^{3}v_{\lambda}}
		&
		\YoungWithData{
			\begin{ytableau}
				0 & \overline{2} & \overline{2} \\
				\overline{2} & \overline{1}
			\end{ytableau}
		}{f_{2}^{5}f_{1}^{4}f_{2}^{4}v_{\lambda}}
		&
		\YoungWithData{
			\begin{ytableau}
				0 & \overline{2} & \overline{2} \\
				\overline{1} & \overline{1}
			\end{ytableau}
		}{f_{2}^{5}f_{1}^{5}f_{2}^{4}v_{\lambda}} \\[12pt]

		\YoungWithData{
			\begin{ytableau}
				0 & \overline{2} & \overline{1} \\
				0 & \overline{1}
			\end{ytableau}
		}{f_{1}f_{2}^{5}f_{1}^{4}f_{2}^{3}v_{\lambda}}
		&
		\YoungWithData{
			\begin{ytableau}
				0 & \overline{2} & \overline{1} \\
				\overline{2} & \overline{1}
			\end{ytableau}
		}{f_{1}f_{2}^{5}f_{1}^{4}f_{2}^{4}v_{\lambda}}
		&
		\YoungWithData{
			\begin{ytableau}
				0 & \overline{2} & \overline{1} \\
				\overline{1} & \overline{1}
			\end{ytableau}
		}{f_{1}f_{2}^{5}f_{1}^{5}f_{2}^{4}v_{\lambda}}
		&
		\YoungWithData{
			\begin{ytableau}
				\overline{2} & \overline{2} & \overline{2} \\
				\overline{1} & \overline{1}
			\end{ytableau}
		}{f_{2}^{6}f_{1}^{5}f_{2}^{4}v_{\lambda}}
		&
		\YoungWithData{
			\begin{ytableau}
				\overline{2} & \overline{2} & \overline{1} \\
				\overline{1} & \overline{1}
			\end{ytableau}
		}{f_{1}f_{2}^{6}f_{1}^{5}f_{2}^{4}v_{\lambda}}
	\end{longtable}
\end{center}
	\end{example}
	
	\section{Linear Independence of the Verma Vectors of $L(\lambda)$}
	In this section, we prove that the Verma vectors of the finite dimensional irreducible representation $L(\lambda)$ of $\mathfrak{spo}(4|1)$ are linearly independent. The proof of linear independence follows the method given in \cite{Raghavan1999}.
	
	We first consider the natural representation $V = \mathbb{C}^{4|1}$ of $\mathfrak{spo}(4|1)$. Let $\{\varepsilon_1, \varepsilon_2, \varepsilon_3, \varepsilon_4, \varepsilon_5\}$ be the standard basis of $V$:
	\[ 
	\varepsilon_1=\begin{pmatrix} 1 \\ 0 \\ 0 \\ 0 \\ 0 \end{pmatrix}, \quad
	\varepsilon_2=\begin{pmatrix} 0 \\ 1 \\ 0 \\ 0 \\ 0 \end{pmatrix}, \quad
	\varepsilon_3=\begin{pmatrix} 0 \\ 0 \\ 1 \\ 0 \\ 0 \end{pmatrix}, \quad
	\varepsilon_4=\begin{pmatrix} 0 \\ 0 \\ 0 \\ 1 \\ 0 \end{pmatrix}, \quad
	\varepsilon_5=\begin{pmatrix} 0 \\ 0 \\ 0 \\ 0 \\ 1 \end{pmatrix}.
	\]
	Here, $V_{\bar{0}} = \operatorname{span}\{\varepsilon_1, \varepsilon_2, \varepsilon_3, \varepsilon_4\}$ and $V_{\bar{1}} = \operatorname{span}\{\varepsilon_5\}$. The highest weight of $V$ is $\epsilon_1$, and $\varepsilon_1$ is the highest weight vector of $V$. The Verma vectors are given by
	\[ 
	\varepsilon_1=\begin{pmatrix} 1 \\ 0 \\ 0 \\ 0 \\ 0 \end{pmatrix}, \quad
	f_1\varepsilon_1=\begin{pmatrix} 0 \\ 1 \\ 0 \\ 0 \\ 0 \end{pmatrix}, \quad
	f_2f_1\varepsilon_1=\begin{pmatrix} 0 \\ 0 \\ 0 \\ 0 \\ 1 \end{pmatrix}, \quad
	f_2^{2}f_1\varepsilon_1=\begin{pmatrix} 0 \\ 0 \\ 0 \\ -1 \\ 0 \end{pmatrix}, \quad
	f_1f_2^{2}f_1\varepsilon_1=\begin{pmatrix} 0 \\ 0 \\ 1 \\ 0 \\ 0 \end{pmatrix},
	\]
	where $f_{1}=E_{21}-E_{34}$ and $f_{2}=E_{52}-E_{45}$. Therefore, the Verma vectors of the natural representation $V$ are linearly independent. We set $\varepsilon_0 \coloneqq \varepsilon_5$, $\varepsilon_{\overline{2}} \coloneqq -\varepsilon_4$, and $\varepsilon_{\overline{1}} \coloneqq \varepsilon_3$. Then the set of Verma vectors of $V$ is $\{\varepsilon_1, \varepsilon_2, \varepsilon_0, \varepsilon_{\overline{2}}, \varepsilon_{\overline{1}}\}$. We have
	\[ 
	f_1\varepsilon_1 = \varepsilon_2, \quad 
	f_2\varepsilon_2 = \varepsilon_0, \quad 
	f_2\varepsilon_0 = \varepsilon_{\overline{2}}, \quad 
	f_1\varepsilon_{\overline{2}} = \varepsilon_{\overline{1}}. 
	\]
	
	We now consider the finite dimensional irreducible representation $L(\lambda)$ of $\mathfrak{spo}(4|1)$ with highest weight $\lambda = (m_1 + m_2)\epsilon_1 + m_2\epsilon_2$. We can realize $L(\lambda)$ as follows. Consider the representation 
	\[ 
	W \coloneqq  V^{\otimes m_1} \otimes (\wedge^{2} V)^{\otimes m_2},
	\] 
	where $\wedge$ denotes the super exterior product, that is, the exterior product on the even part of $V$ and the symmetric product on the odd part of $V$. Then the element 
	\[ 
	v_\lambda \coloneqq \varepsilon_1^{\otimes m_1} \otimes (\varepsilon_1 \wedge \varepsilon_2)^{\otimes m_2} 
	\] 
	of $W$ is a maximal vector, and the $\mathfrak{spo}(4|1)$-submodule of $W$ generated by $v_\lambda$ is a model for $L(\lambda)$.
	
	By weakening the conditions in the definition of the KN tableaux of $\mathfrak{spo}(4|1)$, we define the column-strict Young tableaux of $\mathfrak{spo}(4|1)$ as follows.
	
	\begin{definition}
		A column-strict Young tableau $Y$ of $\mathfrak{spo}(4|1)$ of shape $\lambda$ is a filling of the Young diagram of shape $\lambda$ with entries from $\mathcal{N}=\{1, 2, 0, \overline{2}, \overline{1}\}$, such that the entries in $Y$ are strictly increasing down each column, but $0$ may repeat. Let $\mathrm{CST}_{\lambda}(4|1)$ denote the set of all column-strict Young tableaux of $\mathfrak{spo}(4|1)$ of shape $\lambda$.
	\end{definition}
	
	Each column-strict Young tableau $Y \in \mathrm{CST}_{\lambda}(4|1)$ corresponds to a vector $u(Y) \in W$. The correspondence is defined as follows:
	\[
	Y = \begin{ytableau}
		j_{m_2} & \cdots & j_{1} & i_{m_1}& \cdots & i_1\\
		k_{m_2} & \cdots & k_{1}
	\end{ytableau}
	\, \longmapsto \, u(Y) = \varepsilon_{i_1} \otimes \cdots \otimes \varepsilon_{i_{m_1}} \otimes
	(\varepsilon_{j_1} \wedge \varepsilon_{k_1}) \otimes \cdots \otimes (\varepsilon_{j_{m_2}} \wedge \varepsilon_{k_{m_2}}).
	\]
	For example:
	\[
	Y_1 = \begin{ytableau}
		1 & 0 & \overline{1} \\
		2 & 0
	\end{ytableau}
	\quad \longmapsto \quad u(Y_1) = \varepsilon_{\overline{1}} \otimes (\varepsilon_0 \wedge \varepsilon_0) \otimes (\varepsilon_1 \wedge \varepsilon_2),
	\]
	\[
	Y_2 = \begin{ytableau}
		1 & 0 & \overline{1} \\
		2 & \overline{2}
	\end{ytableau}
	\quad \longmapsto \quad u(Y_2) = \varepsilon_{\overline{1}} \otimes (\varepsilon_0 \wedge \varepsilon_{\overline{2}}) \otimes (\varepsilon_1 \wedge \varepsilon_2).
	\]
	
	\begin{lemma}\label{lem:spo monomials linearly independent}
		Let $K=\{ u(Y) \mid Y \in \mathrm{CST}_{\lambda}(4|1)\}$. Then the vectors in $K$ are linearly independent.
	\end{lemma}
	\begin{proof}
		The set
		\[ \mathcal{B}=
		\left\{
		\varepsilon_{i_1} \otimes \cdots \otimes \varepsilon_{i_{m_1}} \otimes
		(\varepsilon_{j_1} \wedge \varepsilon_{k_1}) \otimes \cdots \otimes (\varepsilon_{j_{m_2}} \wedge \varepsilon_{k_{m_2}})
		\;\middle|\;
		\begin{aligned}
			& i_r, j_s, k_s \in \{1, 2, 0, \overline{2}, \overline{1}\}, \\
			& j_s < k_s \text{ or } j_s = k_s=0, \\
			& r = 1, \dots, m_1,\,s = 1, \dots, m_2
		\end{aligned}
		\right\}
		\]
		forms a basis for the representation space $W$. It is clear that $K\subseteq \mathcal{B}$ (in fact, $K= \mathcal{B}$), and thus the vectors in $K$ are linearly independent.
	\end{proof}
	
	Next, we define a total order on $\mathrm{CST}_{\lambda}(4|1)$. For $Y \in \mathrm{CST}_{\lambda}(4|1)$, let $Y(i,j)$ denote the entry in row $i$ and column $j$ of $Y$. Note that the range of the column index $j$ depends on the row index $i$. We call a pair $(i,j)$ admissible if $Y(i,j)$ makes sense.
	
	\begin{definition}\label{spo total order}
		For admissible pairs $(i,j)$ and $(i',j')$, we say $(i,j) < (i',j')$ if either $j > j'$, or $j = j'$ and $i < i'$. For $Y, Y' \in \mathrm{CST}_{\lambda}(4|1)$ with $Y \neq Y'$, we say $Y < Y'$ if $Y(i,j) < Y'(i,j)$ for the smallest admissible pair $(i,j)$ such that $Y(i,j) \neq Y'(i,j)$. For any $Y, Y' \in \mathrm{CST}_{\lambda}(4|1)$, we say $Y \leq Y'$ if either $Y = Y'$ or $Y < Y'$.
	\end{definition}
	
	Recall the two tableaux $Y_1$ and $Y_2$ in $\mathrm{CST}_{\lambda}(4|1)$ described above. According to Definition \ref{spo total order}, we have $Y_1 \leq Y_2$. Furthermore, the total order defined on $\mathrm{CST}_{\lambda}(4|1)$ naturally extends to the corresponding vectors, so we have $u(Y_1) \leq u(Y_2)$.
	
	\begin{lemma}\label{action of f_2^{b_1}}
		Let $b_1 \leq 2m_2$. The action of $f_2^{b_1}$ on $v_\lambda = \varepsilon_1^{\otimes m_1} \otimes (\varepsilon_1 \wedge \varepsilon_2)^{\otimes m_2}$ is as follows:
		
		\noindent \textnormal{(i)} If $b_1=2k$ is even, where $k \in \mathbb{Z}_{\geq 0}$, then 
		\[ f_2^{2k}v_\lambda = k! \cdot \varepsilon_1^{\otimes m_1} \otimes \sum_{1 \le i_1 < i_2 < \dots < i_k \le m_2} \biggl( \bigotimes_{t=1}^{m_2} w_t^{(i_1,\dots,i_k)} \biggr) ,\]
		where
		\[ 
		w_t^{(i_1,\dots,i_k)} = \begin{cases} 
			\varepsilon_1 \wedge \varepsilon_{\overline{2}}, & t \in \{i_1, \dots, i_k\} ,\\
			\varepsilon_1 \wedge \varepsilon_2, & t \notin \{i_1, \dots, i_k\} .
		\end{cases}
		\]
		
		\noindent \textnormal{(ii)} If $b_1=2k+1$ is odd, where $k \in \mathbb{Z}_{\geq 0}$, then 
		\[ f_2^{2k+1}v_\lambda = k! \cdot \varepsilon_1^{\otimes m_1} \otimes \sum_{1 \le i_1 < \dots < i_k \le m_2} \sum_{j \in \{1, \dots, m_2\} \setminus \{i_1, \dots, i_k\}} \biggl( \bigotimes_{t=1}^{m_2} w_t^{(i_1,\dots,i_k, j)} \biggr), \]
		where
		\[ 
		w_t^{(i_1,\dots,i_k, j)} = \begin{cases} 
			\varepsilon_1 \wedge \varepsilon_{\overline{2}}, &  t \in \{i_1, \dots, i_k\}, \\
			\varepsilon_1 \wedge \varepsilon_0, &  t = j ,\\
			\varepsilon_1 \wedge \varepsilon_2, &  t \notin \{i_1, \dots, i_k\} \cup \{j\}.
		\end{cases}
		\]
	\end{lemma}
	\begin{proof}
		Since $f_2\varepsilon_1=0$ and $|\varepsilon_1|=0$, it suffices to consider the action of $f_2^{b_1}$ on $(\varepsilon_1 \wedge \varepsilon_2)^{\otimes m_2}$ in the proof.
		
		For each $1 \leq p \leq m_2$, define the operator $f^{(p)}_2$ as follows:
		\[ f^{(p)}_2(x_1\otimes\cdots\otimes x_{m_2}) =(-1)^{\sum_{i<p}|x_i|}\; x_1\otimes\cdots\otimes f_2(x_p)\otimes\cdots\otimes x_{m_2}, \]
		where $x_1,\dots, x_{m_2} \in \{\varepsilon_1 \wedge \varepsilon_2, \varepsilon_1 \wedge \varepsilon_0, \varepsilon_1 \wedge \varepsilon_{\overline{2}}\}$, and $|\varepsilon_1\wedge\varepsilon_2|=0$, $|\varepsilon_1\wedge\varepsilon_0|=1$, $|\varepsilon_1\wedge\varepsilon_{\overline{2}}|=0$. 
		
		Setting
		\[ \mathcal{F}\coloneqq\sum_{p=1}^{m_2} f^{(p)}_2, \]
		the action of $f_2^{b_1}$ on $(\varepsilon_1 \wedge \varepsilon_2)^{\otimes m_2}$ is equivalent to that of $\mathcal{F}^{b_1}$. It is straightforward to verify that for $p \neq q$, $f^{(p)}_2$ and $f^{(q)}_2$ satisfy
		\[ f^{(p)}_2 f^{(q)}_2 = -\, f^{(q)}_2 f^{(p)}_2. \]
		Consequently, we have
		\[ \sum_{p\ne q} f^{(p)}_2 f^{(q)}_2 = \tfrac12\sum_{p\ne q}\big(f^{(p)}_2f^{(q)}_2+f^{(q)}_2f^{(p)}_2\big)=0. \]
		Therefore,
		\[ \mathcal{F}^2=\sum_{p,q} f^{(p)}_2 f^{(q)}_2=\sum_{p} f^{(p)}_2 f^{(p)}_2. \]
		
		\noindent (i) We now consider the case where $b_1=2k$ is even. We need to show that for any $k \in \mathbb{Z}_{\geq 0}$, we have
		\[ \mathcal{F}^{2k}\big((\varepsilon_1\wedge\varepsilon_2)^{\otimes m_2}\big)
		= k!\sum_{1\le i_1<\cdots<i_k\le m_2} \biggl( \bigotimes_{t=1}^{m_2} w_t^{(i_1,\dots,i_k)} \biggr), \]
		where
		\[ w_t^{(i_1,\dots,i_k)} = \begin{cases} 
			\varepsilon_1 \wedge \varepsilon_{\overline{2}}, & t \in \{i_1, \dots, i_k\} ,\\
			\varepsilon_1 \wedge \varepsilon_2, & t \notin \{i_1, \dots, i_k\} .
		\end{cases} \]
		
		For $k=0$, the claim holds trivially. For $k=1$, since $\mathcal{F}^2=\sum_{p} f^{(p)}_2 f^{(p)}_2$, we have
		\[ \mathcal{F}^2((\varepsilon_1\wedge\varepsilon_2)^{\otimes m_2})
		= \sum_{i=1}^{m_2}
		(\varepsilon_1\wedge\varepsilon_2)^{\otimes (i-1)}
		\otimes (\varepsilon_1\wedge\varepsilon_{\overline{2}})
		\otimes (\varepsilon_1\wedge\varepsilon_2)^{\otimes (m_2-i)}; \]
		thus the claim holds in this case. Assume that the claim holds for some $k \ge 1$. We consider the case for $k+1$:
		\begin{align*}
			\mathcal{F}^{2(k+1)}\big((\varepsilon_1\wedge\varepsilon_2)^{\otimes m_2}\big) 
			&= \mathcal{F}^2 \left( \mathcal{F}^{2k}\big((\varepsilon_1\wedge\varepsilon_2)^{\otimes m_2}\big) \right) \\
			&= \biggl( \sum_{p=1}^{m_2} f^{(p)}_2 f^{(p)}_2 \biggr) \biggl( k! \sum_{1\le i_1<\cdots<i_k\le m_2} \biggl( \bigotimes_{t=1}^{m_2} w_t^{(i_1,\dots,i_k)} \biggr) \biggr).
		\end{align*}
		Since $\bigotimes_{t=1}^{m_2} w_t^{(i_1,\dots,i_k)}$ has $\varepsilon_1 \wedge \varepsilon_{\overline{2}}$ at positions $i_1,\dots,i_k$ and $\varepsilon_1 \wedge \varepsilon_2$ elsewhere, it follows that
		\[ f^{(p)}_2f^{(p)}_2
		\biggl(\bigotimes_{t=1}^{m_2} w_t^{(i_1,\dots,i_k)}\biggr)
		\neq 0
		\quad \text{if and only if}\quad p\notin\{i_1,\dots,i_k\}. \]
		Consequently, we obtain
		\[ \mathcal{F}^{2}
		\biggl(\bigotimes_{t=1}^{m_2} w_t^{(i_1,\dots,i_k)}\biggr)
		=
		\sum_{p\notin\{i_1,\dots,i_k\}}
		\biggl(\bigotimes_{t=1}^{m_2}
		w_t^{(i_1,\dots,i_k,p)}\biggr), \]
		where
		\[ w_t^{(i_1,\dots,i_k,p)} = \begin{cases} 
			\varepsilon_1 \wedge \varepsilon_{\overline{2}}, & t \in \{i_1, \dots, i_k, p\} ,\\
			\varepsilon_1 \wedge \varepsilon_2, & t \notin \{i_1, \dots, i_k, p\} .\end{cases} \]
		Therefore, 
		\begin{align*}
			\mathcal{F}^{2(k+1)}\big((\varepsilon_1\wedge\varepsilon_2)^{\otimes m_2}\big) &= k!\sum_{1\le i_1<\cdots<i_k\le m_2}\sum_{p\notin\{i_1,\dots,i_k\}}\biggl(\bigotimes_{t=1}^{m_2} w_t^{(i_1,\dots,i_k,p)}\biggr)\\
			&=(k+1)! \sum_{1\le j_1<\cdots<j_{k+1}\le m_2}
			\biggl(\bigotimes_{t=1}^{m_2} w_t^{(j_1,\dots,j_{k+1})}\biggr),
		\end{align*}
		where
		\[ w_t^{(j_1,\dots,j_{k+1})} = \begin{cases} 
			\varepsilon_1 \wedge \varepsilon_{\overline{2}}, & t \in \{j_1, \dots, j_{k+1}\} ,\\
			\varepsilon_1 \wedge \varepsilon_2, & t \notin \{j_1, \dots, j_{k+1}\} .
		\end{cases} \]
		Thus, the claim holds for $k+1$. By mathematical induction, the result holds for all $k \in \mathbb{Z}_{\geq 0}$.
		
		\noindent (ii) Next, we consider the case where $b_1=2k+1$ is odd. By the result of (i), we have
		\begin{align*}
			\mathcal{F}^{2k+1}\big((\varepsilon_1\wedge\varepsilon_2)^{\otimes m_2}\big)
			&= \mathcal{F} \left( \mathcal{F}^{2k}\big((\varepsilon_1\wedge\varepsilon_2)^{\otimes m_2}\big) \right) \\
			&= \mathcal{F} \biggl( k! \sum_{1\le i_1<\cdots<i_k\le m_2} \biggl( \bigotimes_{t=1}^{m_2} w_t^{(i_1,\dots,i_k)} \biggr) \biggr) \\
			&= k! \sum_{1\le i_1<\cdots<i_k\le m_2} \mathcal{F} \biggl( \bigotimes_{t=1}^{m_2} w_t^{(i_1,\dots,i_k)} \biggr),
		\end{align*}
		where $\mathcal{F}=\sum_{j=1}^{m_2} f^{(j)}_2$. Since $\bigotimes_{t=1}^{m_2} w_t^{(i_1,\dots,i_k)}$ has $\varepsilon_1 \wedge \varepsilon_{\overline{2}}$ at positions $i_1,\dots,i_k$ and $\varepsilon_1 \wedge \varepsilon_2$ elsewhere, it follows that
		\[ f^{(j)}_2
		\biggl(\bigotimes_{t=1}^{m_2} w_t^{(i_1,\dots,i_k)}\biggr)
		\neq 0
		\quad \text{if and only if}\quad j\notin\{i_1,\dots,i_k\}. \]
		Consequently, we obtain
		\[ \mathcal{F}^{2k+1}\big((\varepsilon_1\wedge\varepsilon_2)^{\otimes m_2}\big)
		= k! \sum_{1 \le i_1 < \dots < i_k \le m_2} \sum_{j \in \{1, \dots, m_2\} \setminus \{i_1, \dots, i_k\}} \biggl( \bigotimes_{t=1}^{m_2} w_t^{(i_1,\dots,i_k, j)} \biggr), \]
		where
		\[ w_t^{(i_1,\dots,i_k, j)} = \begin{cases}
			\varepsilon_1 \wedge \varepsilon_{\overline{2}}, &  t \in \{i_1, \dots, i_k\}, \\
			\varepsilon_1 \wedge \varepsilon_0, &  t = j ,\\
			\varepsilon_1 \wedge \varepsilon_2, &  t \notin \{i_1, \dots, i_k\} \cup \{j\}.
		\end{cases} \]
		Thus, the claim holds.
	\end{proof}
	
	\begin{proposition}\label{spo largest tableau}
		For the Verma vector $\bm{f^{b}} v_{\lambda}$ of $L(\lambda)$, we have
		\[
		\bm{f^{b}} v_{\lambda} =q(T(\bm{b})) u(T(\bm{b})) + \sum_{Y < T(\bm{b})} q(Y) u(Y),
		\]
		where $T(\bm{b})$ is the KN tableau of $\mathfrak{spo}(4|1)$ corresponding to the Verma vector $\bm{f^{b}} v_{\lambda}$, $Y \in \mathrm{CST}_{\lambda}(4|1)$, and $q(T(\bm{b}))\in \mathbb{Z}_{> 0}$, $q(Y) \in \mathbb{Z}$.
	\end{proposition}
	\begin{proof}
		We first consider the highest weight vector $v_{\lambda}$. The KN tableau $T(\bm{0})$ of $\mathfrak{spo}(4|1)$ corresponding to $v_{\lambda}$ is shown in Figure \ref{fig:spo highest weight}.
		
		\begin{figure}[H]
			\centering
			\begin{tikzpicture}[scale=0.7, every node/.style={font=\large}]
				
				\draw (0,0) rectangle ++(8,1);
				\node at (1,0.5) {1};
				\node at (2,0.5) {1};
				\node at (3,0.5) {$\cdots$};
				\node at (4,0.5) {1};
				\draw (5,0) -- (5,1);
				\node at (5.5,0.5) {1};
				\node at (6.5,0.5) {$\cdots$};
				\node at (7.5,0.5) {1};
				
				\draw (0,-1) rectangle ++(5,1);
				\node at (1,-0.5) {2};
				\node at (2,-0.5) {2};
				\node at (3,-0.5) {$\cdots$};
				\node at (4,-0.5) {2};
				
				\draw[decorate,decoration={brace,amplitude=5pt}] 
				(5,1.2) -- (8,1.2) node[midway,above=3pt] {$m_1$};
				\draw[decorate,decoration={brace,amplitude=5pt,mirror}] 
				(0,-1.5) -- (5,-1.5) node[midway,below=3pt] {$m_2$};
			\end{tikzpicture}
			\caption{The KN tableau of $\mathfrak{spo}(4|1)$ corresponding to the highest weight vector $v_{\lambda}$}
			\label{fig:spo highest weight}
		\end{figure}
		
		Clearly, $T(\bm{0})$ is the smallest tableau in $\mathrm{CST}_{\lambda}(4|1)$. In this case, $v_\lambda = u(T(\bm{0})) = \varepsilon_1^{\otimes m_1} \otimes (\varepsilon_1 \wedge \varepsilon_2)^{\otimes m_2}$, and the claim holds.
		
		For $Y \in \mathrm{CST}_{\lambda}(4|1)$, let $k^{i}_{j}(Y)$ denote the number of entries equal to $j$ in row $i$ of $Y$, where $i \in \{1, 2\}$ and $j \in \mathcal{N}=\{1, 2, 0, \overline{2}, \overline{1}\}$.
		Next, we consider $\bm{f^{b}} v_{\lambda}$, where $\bm{b}=(b_{1}, b_{2}, b_{3}, b_{4})\neq (0,0,0,0)$.
		
		Expressing the result of Lemma \ref{action of f_2^{b_1}} in terms of $\mathrm{CST}_{\lambda}(4|1)$, we obtain
		\begin{align}\label{f_2^{b_1}}
			f_2^{b_1}T(\bm{0})=\lfloor \tfrac{1}{2}b_1 \rfloor! \sum_{U} U,
		\end{align}
		where the sum runs over all tableaux $U \in \mathrm{CST}_{\lambda}(4|1)$ satisfying $k^{1}_{1}(U)=m_1+m_2$, $k^{2}_{\overline{2}}(U)=\lfloor \tfrac{1}{2}b_1 \rfloor$, $k^{2}_{0}(U)=b_1-2\lfloor \tfrac{1}{2}b_1 \rfloor$, and $k^{2}_{2}(U)+k^{2}_{0}(U)+k^{2}_{\overline{2}}(U)=m_2$. Let $T_1$ be the largest tableau among all the summands. If $b_1=2k$ is even, then $T_1$ is given by
		\[
		\begin{tikzpicture}[scale=0.7, baseline=-0.5ex, every node/.style={font=\normalsize}]
			\begin{scope}[xshift=0cm]
				
				\draw (0,0) rectangle ++(7,1);
				\node at (0.5,0.5) {1};
				\node at (1.25,0.5) {$\cdots$};
				\node at (2,0.5) {1};
				\node at (2.75,0.5) {1};
				\node at (3.5,0.5) {$\cdots$};
				\node at (4.25,0.5) {1};
				\node at (5,0.5) {1};
				\node at (5.75,0.5) {$\cdots$};
				\node at (6.5,0.5) {1};
				
				\draw (0,-1) rectangle (4.75,0); 
				\node at (0.5,-0.5) {2};
				\node at (1.25,-0.5) {$\cdots$};
				\node at (2,-0.5) {2};
				\node at (2.75,-0.5) {$\overline{2}$};
				\node at (3.5,-0.5) {$\cdots$};
				\node at (4.25,-0.5) {$\overline{2}$};
				
				\draw[decorate,decoration={brace,mirror,amplitude=4pt}]
				(2.5,-1.1) -- (4.75,-1.1) node[midway,below=2pt] {$k$};
			\end{scope}
		\end{tikzpicture}\;,
		\]
		and if $b_1=2k+1$ is odd, then $T_1$ is given by
		\[
		\begin{tikzpicture}[scale=0.7, baseline=-0.5ex, every node/.style={font=\normalsize}]
			\begin{scope}[xshift=0cm]
				
				\draw (0,0) rectangle ++(7,1);
				\node at (0.4,0.5) {1};
				\node at (1.1,0.5) {$\cdots$};
				\node at (1.8,0.5) {1};
				\node at (2.35,0.5) {1};
				\node at (2.9,0.5) {1};
				\node at (3.65,0.5) {$\cdots$};
				\node at (4.35,0.5) {1};
				\node at (5,0.5) {1};
				\node at (5.75,0.5) {$\cdots$};
				\node at (6.5,0.5) {1};
				
				\draw (0,-1) rectangle (4.75,0); 
				\node at (0.4,-0.5) {2};
				\node at (1.1,-0.5) {$\cdots$};
				\node at (1.8,-0.5) {2};
				\node at (2.35,-0.5) {0};
				\node at (2.9,-0.5) {$\overline{2}$};
				\node at (3.65,-0.5) {$\cdots$};
				\node at (4.35,-0.5) {$\overline{2}$};
				
				\draw[decorate,decoration={brace,mirror,amplitude=4pt}]
				(2.7,-1.1) -- (4.75,-1.1) node[midway,below=2pt] {$k$};
			\end{scope}
		\end{tikzpicture}\;.
		\]
		Clearly, $T_1$ is the KN tableau of $\mathfrak{spo}(4|1)$ corresponding to $f_2^{b_1}v_\lambda$. By \eqref{f_2^{b_1}}, we have
		\[ f_2^{b_1}T(\bm{0})=\lfloor \tfrac{1}{2}b_1 \rfloor! \,T_1+\lfloor \tfrac{1}{2}b_1 \rfloor! \sum_{U<T_1} U. \]
		
		Consider $f_1^{b_2}f_2^{b_1}T(\bm{0})$:
		\begin{align}\label{f_1^{b_2}f_2^{b_1}}
			f_1^{b_2}f_2^{b_1}T(\bm{0})=\lfloor \tfrac{1}{2}b_1 \rfloor! \,\left( f_1^{b_2}T_1\right) +\lfloor \tfrac{1}{2}b_1 \rfloor! \sum_{U<T_1} \left( f_1^{b_2}U\right) .
		\end{align}
		Let $T_2$ be the largest tableau in the expansion of $f_1^{b_2}f_2^{b_1}T(\bm{0})$. Then $T_2$ appears only in the expansion of $f_1^{b_2}T_1$ and does not appear in any $f_1^{b_2}U$. Furthermore, the coefficient of $T_2$ in $f_1^{b_2}T_1$ is $b_2!$. It is easy to see that $T_2$ is the KN tableau of $\mathfrak{spo}(4|1)$ corresponding to $f_1^{b_2}f_2^{b_1}v_\lambda$. Write
		\[ f_1^{b_2}T_1=b_2!\,T_2+\sum_{X'<T_2}h'(X')X', \]
		where $h'(X')\in \mathbb{Z}_{\geq 0}$, and each $X'$ is a tableau in $\mathrm{CST}_{\lambda}(4|1)$ strictly smaller than $T_2$. Substituting the expression for $f_1^{b_2}T_1$ into \eqref{f_1^{b_2}f_2^{b_1}} yields
		\begin{align*}
			f_1^{b_2}f_2^{b_1}T(\bm{0})&=\lfloor \tfrac{1}{2}b_1 \rfloor! \,\biggl( b_2!\,T_2+\sum_{X'<T_2}h'(X')X'\biggr) +\lfloor \tfrac{1}{2}b_1 \rfloor! \sum_{U<T_1} \left( f_1^{b_2}U\right) \\
			&=h_2T_2+\sum_{X<T_2}h(X)X , 
		\end{align*}
		where $h_2=\lfloor \tfrac{1}{2}b_1 \rfloor! b_2!$, $h(X)\in \mathbb{Z}_{\geq 0}$, and each $X$ is a tableau in $\mathrm{CST}_{\lambda}(4|1)$ strictly smaller than $T_2$.
		
		Consider $f_2^{b_3}f_1^{b_2}f_2^{b_1}T(\bm{0})$:
		\begin{align}\label{f_2^{b_3}f_1^{b_2}f_2^{b_1}}
			f_2^{b_3}f_1^{b_2}f_2^{b_1}T(\bm{0})=h_2\left( f_2^{b_3}T_2\right) +\sum_{X<T_2}h(X)\left( f_2^{b_3}X\right) .
		\end{align}
		Let $T_3$ be the largest tableau in the expansion of $f_2^{b_3}f_1^{b_2}f_2^{b_1}T(\bm{0})$. Then $T_3$ appears only in the expansion of $f_2^{b_3}T_2$ and does not appear in any $f_2^{b_3}X$. Furthermore, the coefficient of $T_3$ in $f_2^{b_3}T_2$ is $\lfloor \tfrac{1}{2}b_3 \rfloor!$. It is easy to see that $T_3$ is the KN tableau of $\mathfrak{spo}(4|1)$ corresponding to $f_2^{b_3}f_1^{b_2}f_2^{b_1}v_\lambda$. Write
		\[ f_2^{b_3}T_2=\lfloor \tfrac{1}{2}b_3 \rfloor!\, T_3+\sum_{Z'<T_3}p'(Z')Z', \]
		where $p'(Z')\in \mathbb{Z}$, and each $Z'$ is a tableau in $\mathrm{CST}_{\lambda}(4|1)$ strictly smaller than $T_3$. Substituting the expression for $f_2^{b_3}T_2$ into \eqref{f_2^{b_3}f_1^{b_2}f_2^{b_1}} yields
		\begin{align*}
			f_2^{b_3}f_1^{b_2}f_2^{b_1}T(\bm{0})&=h_2\biggl( \lfloor \tfrac{1}{2}b_3 \rfloor!\, T_3+\sum_{Z'<T_3}p'(Z')Z'\biggr) +\sum_{X<T_2}h(X)\left( f_2^{b_3}X\right)\\
			&=h_3T_3+\sum_{Z<T_3}p(Z)Z , 
		\end{align*}
		where $h_3=h_2 \lfloor \tfrac{1}{2}b_3 \rfloor!$, $p(Z)\in \mathbb{Z}$, and each $Z$ is a tableau in $\mathrm{CST}_{\lambda}(4|1)$ strictly smaller than $T_3$.
		
		Finally, consider $f_1^{b_4}f_2^{b_3}f_1^{b_2}f_2^{b_1}T(\bm{0})$:
		\begin{align}\label{b}
			\bm{f^{b}}T(\bm{0})=h_3\left( f_1^{b_4}T_3\right) +\sum_{Z<T_3}p(Z)\left( f_1^{b_4}Z\right)  .
		\end{align}
		Let $T_4$ be the largest tableau in the expansion of $\bm{f^{b}}T(\bm{0})$. Then $T_4$ appears only in the expansion of $f_1^{b_4}T_3$ and does not appear in any $f_1^{b_4}Z$. Furthermore, the coefficient of $T_4$ in $f_1^{b_4}T_3$ is $b_4!$. It is easy to see that $T_4$ is exactly $T(\bm{b})$, the KN tableau of $\mathfrak{spo}(4|1)$ corresponding to $\bm{f^{b}}v_\lambda$. Write
		\[ f_1^{b_4}T_3=b_4!\,T(\bm{b})+\sum_{Y'<T(\bm{b})}q'(Y')Y', \]
		where $q'(Y')\in \mathbb{Z}_{\geq 0}$, and each $Y'$ is a tableau in $\mathrm{CST}_{\lambda}(4|1)$ strictly smaller than $T(\bm{b})$. Substituting the expression for $f_1^{b_4}T_3$ into \eqref{b} yields
		\begin{align*}
			\bm{f^{b}}T(\bm{0})&=h_3\biggl( b_4!\,T(\bm{b})+\sum_{Y'<T(\bm{b})}q'(Y')Y'\biggr) +\sum_{Z<T_3}p(Z)\left( f_1^{b_4}Z\right)  \\
			&=q(T(\bm{b}))T(\bm{b})+\sum_{Y<T(\bm{b})}q(Y)Y , 
		\end{align*}
		where $q(T(\bm{b}))=h_3b_4!$, $q(Y)\in \mathbb{Z}$, and each $Y$ is a tableau in $\mathrm{CST}_{\lambda}(4|1)$ strictly smaller than $T(\bm{b})$. Therefore, we have
		\[ \bm{f^{b}}v_\lambda=q(T(\bm{b}))u(T(\bm{b}))+\sum_{Y<T(\bm{b})}q(Y)u(Y). \]
		Thus, the proposition holds.
	\end{proof}
	
	\begin{theorem}
		All Verma vectors of the irreducible representation $L(\lambda)$ of $\mathfrak{spo}(4|1)$ are linearly independent. 
	\end{theorem}
	\begin{proof}
		Let the set of KN tableaux of $\mathfrak{spo}(4|1)$ be given by
		\[ \mathrm{KN}_{\lambda}(4|1)=\left\lbrace T(\bm{b_1})> T(\bm{b_2})> \cdots> T(\bm{b_r}) \right\rbrace. \]
		Then the corresponding set of Verma vectors is
		\[ H =\left\lbrace \bm{f^{b_1}} v_\lambda, \bm{f^{b_2}} v_\lambda, \cdots, \bm{f^{b_r}} v_\lambda \right\rbrace. \]
		
		Suppose there exist $x_1, x_2, \ldots, x_r \in \mathbb{C}$ such that
		\[
		x_1 \bm{f^{b_1}} v_\lambda + x_2 \bm{f^{b_2}} v_\lambda + \dots + x_r \bm{f^{b_r}} v_\lambda = 0.
		\]
		Then by Proposition \ref{spo largest tableau}, we have
		\begin{align*}
			&\quad x_1\biggl( q(T(\bm{b_1}))u(T(\bm{b_1})) + \sum_{Y_1 < T(\bm{b_1})} q(Y_1)\, u(Y_1)\biggr) \\
			&+ x_2\biggl( q(T(\bm{b_2}))u(T(\bm{b_2})) + \sum_{Y_2 < T(\bm{b_2})} q(Y_2)\, u(Y_2)\biggr) \\
			& + \cdots + x_r\biggl( q(T(\bm{b_r}))u(T(\bm{b_r})) + \sum_{Y_r < T(\bm{b_r})} q(Y_r)\, u(Y_r)\biggr) = 0,
		\end{align*}
		where $q(T(\bm{b_1})), q(T(\bm{b_2})), \dots, q(T(\bm{b_r})) \in \mathbb{Z}_{> 0}$.
		Since $T(\bm{b_1}) > T(\bm{b_2}) > \cdots > T(\bm{b_r})$, it follows that
		\[
		u(T(\bm{b_1})) > u(T(\bm{b_2})) > \cdots > u(T(\bm{b_r})).
		\]
		Combined with Lemma \ref{lem:spo monomials linearly independent}, we deduce that $x_1 = 0$. Now assume that for some $i \in \left\lbrace 2, 3, \ldots, r \right\rbrace$, we have $x_1 = x_2 = \cdots = x_{i-1} = 0$. Then
		\begin{align*}
			&\quad x_i\biggl( q(T(\bm{b_i}))u(T(\bm{b_i})) + \sum_{Y_i < T(\bm{b_i})} q(Y_i)\, u(Y_i)\biggr)\\
			& + x_{i+1}\biggl( q(T(\bm{b_{i+1}}))u(T(\bm{b_{i+1}})) + \sum_{Y_{i+1} < T(\bm{b_{i+1}})} q(Y_{i+1})\, u(Y_{i+1})\biggr) \\
			& + \cdots 
			+ x_r\biggl( q(T(\bm{b_r}))u(T(\bm{b_r})) + \sum_{Y_r < T(\bm{b_r})} q(Y_r)\, u(Y_r)\biggr) = 0.
		\end{align*}
		Since $u(T(\bm{b_i})) > u(T(\bm{b_{i+1}})) > \cdots > u(T(\bm{b_r}))$, we obtain $x_i = 0$. By mathematical induction, we have
		\[
		x_1 = x_2 = \cdots = x_r = 0.
		\]
		
		Therefore, all the Verma vectors are linearly independent. 
	\end{proof}
	
	By Theorem \ref{KN-correspondence}, there is a one-to-one correspondence between the set $H$ of Verma vectors of $L(\lambda)$ and the set $\mathrm{KN}_{\lambda}(4|1)$ of KN tableaux of $\mathfrak{spo}(4|1)$, hence
	\[
	|H| = |\mathrm{KN}_{\lambda}(4|1)|.
	\]
	Combined with Proposition \ref{dimension}, we obtain
	\[
	\dim L(\lambda) = |H|.
	\]
	Since the Verma vectors in $H$ are linearly independent, $H$ forms the Verma basis of $L(\lambda)$. We summarize this result in the following theorem.
	
	\begin{theorem}\label{Verma basis}
		The set $H$ of Verma vectors is the Verma basis of the finite dimensional irreducible representation $L(\lambda)$ of $\mathfrak{spo}(4|1)$.
	\end{theorem}
	
	\begin{example}
		All the Verma vectors in Example \textnormal{\ref{example}} form the Verma basis of $L(\lambda)$ with the highest weight $\lambda = 3\epsilon_1 + 2\epsilon_2$.
	\end{example}

\section*{Acknowledgments}
The authors are partially supported by NSFC (Grant No. 12161090).

\end{document}